\documentclass{amsart}
\usepackage{hyperref}
\usepackage{tikz}
\usepackage{blkarray} 
\usepackage{multirow}
\usepackage{amssymb}
\usepackage{enumerate}

\usetikzlibrary{decorations.pathmorphing}
\usetikzlibrary{decorations.markings}

\hypersetup{
pdftitle={Excluded vertex-minors for graphs of linear rank-width at most k},
pdfauthor={Jisu Jeong, O-joung Kwon, and Sang-il Oum}
}
\newcommand\abs[1]{\lvert #1\rvert}
\newtheorem{THM}{Theorem}[section]
\newtheorem{LEM}[THM]{Lemma}
\newtheorem{COR}[THM]{Corollary}
\newtheorem{PROP}[THM]{Proposition}
\newtheorem*{thmext}{Theorem \ref{thm:blockbou}}
\newcommand\rank{\operatorname{rank}}

\newcommand\lrw{\operatorname{lrw}}

\newcommand\cutrk{\rho}%
\newcommand\pivot\wedge
\newcommand\Aut{\operatorname{Aut}}
\newcommand\norb{\operatorname{\#Orb}}
\theoremstyle{definition}

\newtheorem{QUE}{Question}
\begin{document}
\title{Excluded vertex-minors for graphs of linear rank-width at most
  $k$}
\thanks{An extended abstract appeared in 
  Proc. 30th International Symposium on Theoretical Aspects of
  Computer Science, 2013 \cite{JKO2013}.}
\address{Department of Mathematical Sciences, KAIST, 291 Daehak-ro
  Yuseong-gu Daejeon, 305-701 South Korea}
\author{Jisu Jeong} 
\email{jjisu@kaist.ac.kr}
\author{O-joung Kwon} 
\email{ojoung@kaist.ac.kr}
\author{Sang-il Oum} 
\email{sangil@kaist.edu}
\thanks{Supported by Basic Science Research
  Program through the National Research Foundation of Korea (NRF)
  funded by the Ministry of Education, Science and Technology
  (2011-0011653). S. O. is also supported by TJ Park Junior Faculty Fellowship.}
\date{\today}

\begin{abstract}
Linear rank-width is a graph width parameter, which is a variation of
rank-width by restricting its tree to a caterpillar.
As a corollary of known theorems,
for each $k$, 
there is a finite obstruction set $\mathcal{O}_k$ of graphs such that a graph $G$ has linear
rank-width at most $k$ if and only if no vertex-minor of $G$ is
isomorphic to a graph in $\mathcal{O}_k$.
However,  no attempts have been made to bound the number of graphs in
$\mathcal{O}_k$ for $k\ge 2$.
We show that for each $k$, there are at least $2^{\Omega(3^k)}$ pairwise locally
non-equivalent graphs in $\mathcal{O}_k$,
and therefore the number of graphs in $\mathcal{O}_k$ is at least
double exponential.

To prove this theorem, 
it is necessary to characterize when two graphs in $\mathcal O_k$ are
locally equivalent.
A graph is a \emph{block graph} if all of its blocks are complete graphs.
We prove that if two block graphs without simplicial vertices of degree
at least $2$ are locally
equivalent, then they are isomorphic. 
This not only is useful for our theorem but also implies a theorem of
Bouchet
 [\emph{Transforming trees by successive local complementations}, J.
  Graph Theory \textbf{12} (1988), no.~2, 195--207]
stating that if two trees are locally equivalent, then they are isomorphic. 

\end{abstract}

\maketitle

\section{Introduction} \label{sec:intro}

Linear rank-width is a width parameter of graphs motivated by
rank-width of graphs introduced by Oum and Seymour~\cite{OS2004}.
A vertex-minor relation is a graph containment relation such that rank-width and linear rank-width cannot increase when taking vertex-minors of a graph. Two graphs $G$, $H$ are called \emph{locally equivalent} if $H$ is a vertex-minor of $G$ and $\abs{V(H)}=\abs{V(G)}$. The definitions can be found in Section~\ref{sec:prelim}.

Oum~\cite{Oum082} proved that for every infinite sequence $G_1$, $G_2,
\ldots$ of graphs of bounded rank-width, there exist $i<j$ such that $G_i$ is isomorphic to a vertex-minor of $G_j$. As a corollary, we immediately obtain the following theorem.

\begin{THM}[Oum~\cite{Oum082}]\label{thm:vertexminorwqo}
For every class $\mathcal{C}$ of graphs of bounded
rank-width, there is a finite list of graphs $G_1$, $G_2, \ldots, G_m$
such that a graph is in $\mathcal C$ if and only if it does not have a
vertex-minor isomorphic to $G_i$ for some $i\in \{1,2, \ldots, m\}$.
\end{THM}

Because rank-width is always less than or equal to linear rank-width, we deduce the following.

\begin{COR}\label{cor:vertexminorwqo2}
For a fixed $k$, there exists a finite set $\mathcal{O}_k$ of graphs $G_1,
G_2, \ldots, G_m$ such that 
a graph has linear rank-width at most $k$ if and only if 
it does not have a vertex-minor isomorphic to $G_i$ for some $i\in \{1,2, \ldots, m\}$.
\end{COR}

However, Theorem~\ref{thm:vertexminorwqo} does not
produce an explicit upper or lower bound on the number of graphs in
$\mathcal O_k$ for Corollary~\ref{cor:vertexminorwqo2}. We aim to
prove a lower bound on $\abs{\mathcal O_k}$.
Our main result is the following.

\begin{THM}\label{thm:bigmain}
Let $k\geq 2$ be an integer. There exist at least $2^{\Omega(3^k)}$
pairwise locally non-equivalent graphs that are vertex-minor minimal
with the property that they have linear rank-width larger than $k$. 
In other words, $\abs{\mathcal O_k}\ge 2^{\Omega(3^k)}$ in Corollary~\ref{cor:vertexminorwqo2}.
\end{THM}

\begin{figure}[t]\centering
\tikzstyle{v}=[circle, draw, solid, fill=black, inner sep=0pt, minimum width=2.5pt]
\tikzset{photon/.style={decorate, decoration={snake}}}
\begin{tikzpicture}[scale=0.05]

\draw (30,55) -- (50,40) -- (40,20) -- (20,20) -- (10,40) -- (30,55);

\node [v] at (30,55) {};
\node [v] at (50,40) {};
\node [v] at (40,20) {};
\node [v] at (20,20) {};
\node [v] at (10,40) {};

\end{tikzpicture}\qquad
\begin{tikzpicture}[scale=0.047]

\draw (30,50) -- (30,60);
\draw (30,50) -- (15,30) -- (45,30) --(30,50);
\draw (15,30) -- (5,20);
\draw (45,30) -- (55,20);

\node [v] at (30,60) {};
\node [v] at (30,50) {};
\node [v] at (15,30) {};
\node [v] at (45,30) {};
\node [v] at (5,20) {};
\node [v] at (55,20) {};

\end{tikzpicture}\qquad
\begin{tikzpicture}[scale=0.053]

\draw (30,40) -- (40,25) -- (30,10) --(20,25) -- (30,40);
\draw (40,25) -- (50,25);
\draw (20,25) -- (10,25);

\node [v] at (30,40) {};
\node [v] at (40,25) {};
\node [v] at (30,10) {};
\node [v] at (20,25) {};
\node [v] at (50,25) {};
\node [v] at (10,25) {};

\end{tikzpicture}
\caption{Graphs in $\mathcal{O}_1$.}
\label{fig:lrw1}
\end{figure}
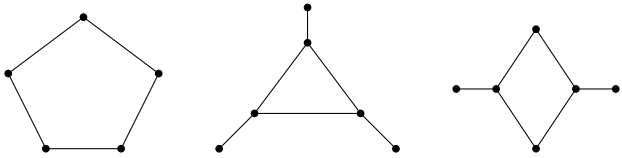

It is non-trivial to characterize the set of all graphs of linear rank-width at most
$k$ in terms of forbidden
vertex-minors. So far only
one case is known. 
For $k=1$, Adler, Farley, and Proskurowski~\cite{AFP2013} characterized
the graphs of linear rank-width at most $1$ by a set $\mathcal{O}_1$
of three graphs in Figure~\ref{fig:lrw1}. 
Ganian~\cite{Ganian10} described the structure of graphs of linear
rank-width $1$.

\begin{table} \label{table}
\begin{center}
\begin{tabular}[t]{c| c |c |c}
width parameter & relation & type & references \\ \hline
  path-width & minor & L, U & \cite{TASY94},\cite{LAG98}  \\
linear-width & minor & L & \cite{Thilikos00}  \\ 
 tree-width & minor & L, U & \cite{GKS99}, \cite{LAG98} \\ 
 tree-depth & minor, induced subgraph & L, U & \cite{DAT12}  \\ 
 rank-width & vertex-minor & U & \cite{Oum05}  \\ 
 branch-width (graphs, matroids) & minor & U & \cite{GGRW03}  
\end{tabular}
\end{center}
\caption{Known lower or upper bound of the size of the obstruction set
  for graphs of bounded width parameters. In the column of type, L and U
  mean a lower
  and upper bound, respectively.}
\end{table}

	There have been similar results on the number of forbidden minors for
various graph width parameters, see Table~\ref{table}.

	One of the main ingredients is a generalization of a theorem
        of Bouchet.
    To show Theorem~\ref{thm:bigmain}, 
	we will construct, for each non-negative integer $k$, a set $\Delta_k$
	of vertex-minor minimal graphs with the property that they have linear rank-width larger than $k$. 
	To obtain the lower bound on $\abs {\mathcal O_k}$, it is
        necessary to understand when two graphs in $\Delta_k$ are
        locally equivalent.
	We resolve this problem by showing the following stronger theorem.
        A vertex is \emph{simplicial} if the set of its neighbors is a clique.
	\begin{THM}\label{thm:blockbou}
	If two block graphs without simplicial vertices of degree 
        at least $2$ are locally equivalent, 
	then they are isomorphic. 
	\end{THM}
        All graphs in $\Delta_k$ have %
        no simplicial vertices of degree at least $2$.
        Hence, Theorem~\ref{thm:blockbou} is useful for proving
        Theorem~\ref{thm:bigmain}.
        Since trees are block graphs without simplicial vertices of
        degree at least $2$, we
        deduce the following corollary, originally shown by Bouchet.
	\begin{COR}[Bouchet~\cite{Bouchet88}]\label{cor:treelocal}
          If two trees are locally equivalent, %
          then they are isomorphic.
	\end{COR}

	The paper is organized as follows. 
	In Section~\ref{sec:prelim}, we present necessary definitions. 
	In Section~\ref{sec:minimalvm}, we construct the set $\Delta_k$ 
	and prove that the graphs in $\Delta_k$ are vertex-minor minimal graphs 
with the property that they have linear rank-width larger than $k$.	  
	In Section~\ref{sec:notlocal}, we prove that no two non-isomorphic graphs in $\Delta_k$ are locally equivalent by showing Theorem~\ref{thm:blockbou}. 
	In Section~\ref{sec:size}, we count graphs in $\Delta_k$ up to
        isomorphism, and we conclude that $\abs{\mathcal{O}_k}\geq 2^{\Omega(3^k)}$. 
	Final remarks are made in Section~\ref{sec:final}.

\section{Preliminaries} \label{sec:prelim}

      In this paper, graphs have no loops and parallel edges. 
      Let $G=(V(G),E(G))$ be a graph with the vertex set $V(G)$ and the edge set $E(G)$. 
      For $S\subseteq V(G)$, $G[S]$ denotes the subgraph of $G$ induced on $S$. %
      And for $v\in V(G)$, we denote $N_G(v)$ as the set of the neighbors of $v$ in $G$. 
A vertex $v$ in $G$ is a \emph{leaf} if $\abs{N_G(v)}=1$. 

For an $X\times Y$ matrix $M=(m_{i,j})_{i\in X, j\in Y}$ and subsets $A\subseteq X$ and $B\subseteq Y$, $M[A,B]$ denotes the $A\times B$ submatrix $(m_{i,j})_{i\in A,j\in B}$ of $M$. 

\subsection*{Vertex-minors.} The \emph{local complementation} at a vertex $v$ of a graph $G=(V,E)$ is an operation to obtain a graph $G* v$ from $G$ by replacing the subgraph $G[N_G(v)]$ with the complementary subgraph of $G[N_G(v)]$. The graph obtained from $G$ by \emph{pivoting} an edge $uv$ is defined by $G \pivot uv=G*u*v*u$.

\begin{figure}[t]\centering
\tikzstyle{v}=[circle, draw, solid, fill=black, inner sep=0pt, minimum width=2.5pt]
\tikzset{photon/.style={decorate, decoration={snake}}}
\begin{tikzpicture}[scale=0.05]

\node [v] (a) at (10,50) {};
\node [v] (b) at (30,50) {};
\node [v] (c) at (0,30) {};
\node [v] (d) at (10,15) {};
\node [v] (e) at (30,17) {};
\node [v] (f) at (40,35) {};
\node [v] (g) at (40,10) {};

\draw (c) -- (a) -- (b) -- (f);
\draw (d) -- (a) -- (e) -- (b);
\draw (f) -- (c) -- (d);
\draw (c)-- (e);

\draw (d) -- (g) -- (f);

\draw (10,55) node{$a$};
\draw (30,55) node{$b$};
\draw (0,25) node{$c$};
\draw (10,10) node{$d$};
\draw (35,14) node{$e$};
\draw (45,35) node{$f$};
\draw (45,15) node{$g$};

\end{tikzpicture}\qquad\quad
\begin{tikzpicture}[scale=0.05]

\node [v] (a) at (10,50) {};
\node [v] (b) at (30,50) {};
\node [v] (c) at (0,30) {};
\node [v] (d) at (10,15) {};
\node [v] (e) at (30,17) {};
\node [v] (f) at (40,35) {};
\node [v] (g) at (40,10) {};

\draw (c) -- (a) -- (b) -- (f);
\draw (d) -- (a) -- (e) -- (b);
\draw (c) -- (d);

\draw (d) -- (e) --(f);
\draw (d) -- (f);

\draw (d) -- (g) -- (f);

\draw (10,55) node{$b$};
\draw (30,55) node{$a$};
\draw (0,25) node{$c$};
\draw (10,10) node{$d$};
\draw (35,14) node{$e$};
\draw (45,35) node{$f$};
\draw (45,15) node{$g$};

\end{tikzpicture}\caption{Pivoting an edge $ab$.}\label{fig:pivoting}
\end{figure}
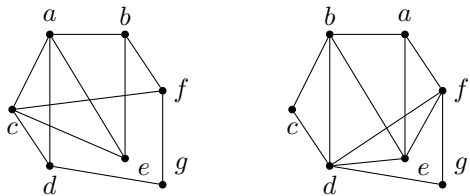

To see how we obtain the resulting graph by pivoting an edge $uv$, let $V_1=N_G(u)\cap N_G(v)$, $V_2=N_G(u)\backslash N_G(v)\backslash \{v\}$, and $V_3=N_G(v)\backslash N_G(u)\backslash \{u\}$. One can easily verify that $G\pivot uv$ is identical to the graph obtained from $G$ by complementing adjacency of vertices between distinct sets $V_i$ and $V_j$, and swapping the vertices $u$ and $v$~\cite{Oum05}. See Figure~\ref{fig:pivoting} for an example.

A graph $H$ is a \emph{vertex-minor} of $G$ if $H$ can be obtained from $G$ by applying a sequence of vertex deletions and local complementations. A graph $H$ is \emph{locally equivalent} to $G$ if $H$ can be obtained from $G$ by applying a sequence of local complementations.

      A vertex-minor $H$ of $G$ is \emph{elementary} if $\abs{V(H)}=\abs{V(G)}-1$. 
      For a set $\mathcal{C}$ of graphs closed under taking vertex-minors, a graph $G$ is an \emph{excluded vertex-minor} for $C$ if $G\notin \mathcal{C}$ and $H\in \mathcal{C}$ for every elementary vertex-minor $H$ of $G$.

\subsection*{Linear rank-width.} 
The adjacency matrix of a graph $G$, which is a $(0,1)$-matrix over the binary field, will be denoted by $A(G)$.
The \emph{cut-rank} function $\cutrk_{G} :  2^{V} \rightarrow \mathbb{Z}$ of a graph $G$ is defined by 
$\cutrk_{G}(X)=\rank(A(G)[X,V\backslash X])$.
The cut-rank function satisfies the submodular inequality \cite{OS2004}:
\[\cutrk_G(X)+\cutrk_G(Y)\ge\cutrk_G(X\cap Y)+\cutrk_G(X\cup Y).\]
This implies that $\cutrk_G(X\cup\{v\})\le \cutrk_G(X)+1$ for
$X\subseteq V(G)$ and $v\in V(G)$.

	 A \emph{linear layout} $L$ of $G$ is a sequence $(v_1, v_2, \ldots, v_{\abs{V(G)}})$ of $V(G)$. 
	 For a linear layout $L$ of $G$ and $a,b\in V(G)$, we denote $a\leq_{L} b$ or $b\ge_L a$ if $a=b$ or $a$ appears before $b$ in $L$. 
	 For two sequences $L_1=(v_1, v_2, \ldots, v_n)$ and $L_2=(w_1, w_2, \ldots, w_m)$, we define $L_1\oplus L_2=(v_1, v_2, \ldots, v_n, w_1, w_2, \ldots, w_m)$.

The \emph{width} of a linear layout $L$ in $G$, denoted by
$\lrw_L(G)$, is defined as the maximum over all $\cutrk_G(\{w:w\leq_L
v\})$ for $v\in V(G)$. 
We say that the width of $L$ is $0$ if $\abs{V(G)}\le 1$.
The \emph{linear rank-width} of $G$, denoted by $\lrw(G)$, is the
minimum width of all linear layouts of $G$.

If two graphs are locally equivalent, then they have the same linear
rank-width 
by the following proposition.
\begin{PROP}[Bouchet~\cite{Bouchet1989a}; see Oum~\cite{Oum05}]\label{prop:leq}
  Let $G$ be a graph and $v\in V(G)$. Then 
  $ \cutrk_G(X)=\cutrk_{G* v}(X)$ for all $X\subseteq V(G)$.
\end{PROP}  
It follows easily that if $H$ is a vertex-minor of $G$, then
$\lrw(H)\le \lrw(G)$.

\section{Excluded vertex-minors for graphs of bounded linear rank-width} \label{sec:minimalvm}

      To prove Theorem~\ref{thm:bigmain}, 
      we construct a set $\Delta_k$ 
      of graphs that are vertex-minor minimal with the property that
      the linear rank-width is larger than $k$. %

A \emph{delta composition} $G$ of graphs $G_1$, $G_2$, and $G_3$ is a graph obtained from the disjoint union of $G_1$, $G_2$, and $G_3$ by adding a triangle $v_1v_2v_3$ where $v_i\in V(G_i)$ for $i=1$, $2$, $3$. We call $v_1v_2v_3$ the \emph{main triangle} of $G$. For a non-negative integer $k$, we define $\Delta_k$ as follows:
\begin{enumerate}
\item $\Delta_0=\{ (\{x,y\}, \{xy\} )  \}.$ (It is isomorphic to $K_2$.)
\item For $i\geq 1$, $\Delta_i$ is the set of all delta compositions of three graphs in $\Delta_{i-1}$.
\end{enumerate}

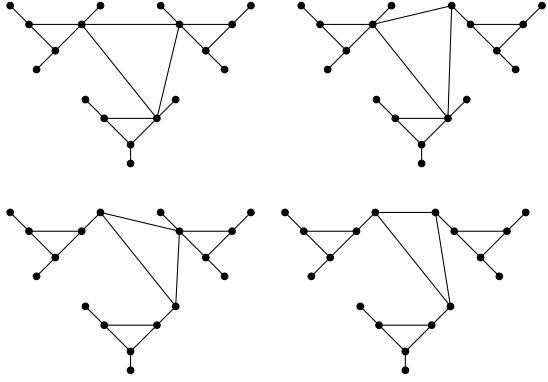
\begin{figure}\centering
 \tikzstyle{v}=[circle, draw, solid, fill=black, inner sep=0pt, minimum width=2.5pt]
\begin{tikzpicture}[scale=0.05]

\draw (8,45) -- (22,45) -- (15, 38) -- (8, 45);

\draw (28,20) -- (42,20) -- (35, 13) -- (28, 20);

\draw (48,45) -- (62,45) -- (55, 38) -- (48, 45);

\draw (8,45) -- (3, 50);
\draw (22,45) -- (27, 50);
\draw (15,38) -- (10, 33);
\draw (28,20) -- (23, 25);
\draw (42,20) -- (47, 25);
\draw (35,13) -- (35, 8);
\draw (48,45) -- (43, 50);
\draw (62,45) -- (67, 50);
\draw (55,38) -- (60, 33);

\draw (22,45) -- (42,20) -- (48,45) -- (22,45);

\node [v] at (8,45) {};
\node [v] at (3,50) {};
\node [v] at (22,45) {};
\node [v] at (27,50) {};
\node [v] at (15,38) {};
\node [v] at (10,33) {};
\node [v] at (28,20) {};
\node [v] at (23,25) {};
\node [v] at (42,20) {};
\node [v] at (47,25) {};
\node [v] at (35,13) {};
\node [v] at (35,8) {};

\node [v] at (48,45) {};
\node [v] at (43,50) {};
\node [v] at (62,45) {};
\node [v] at (67,50) {};
\node [v] at (55,38) {};
\node [v] at (60,33) {};

\node at (0,-1) {};

\end{tikzpicture}\quad
\begin{tikzpicture}[scale=0.05]

\draw (8,45) -- (22,45) -- (15, 38) -- (8, 45);

\draw (28,20) -- (42,20) -- (35, 13) -- (28, 20);

\draw (48,45) -- (62,45) -- (55, 38) -- (48, 45);

\draw (8,45) -- (3, 50);
\draw (22,45) -- (27, 50);
\draw (15,38) -- (10, 33);
\draw (28,20) -- (23, 25);
\draw (42,20) -- (47, 25);
\draw (35,13) -- (35, 8);
\draw (48,45) -- (43, 50);
\draw (62,45) -- (67, 50);
\draw (55,38) -- (60, 33);

\draw (22,45) -- (42,20) -- (43,50) -- (22,45);

\node [v] at (8,45) {};
\node [v] at (3,50) {};
\node [v] at (22,45) {};
\node [v] at (27,50) {};
\node [v] at (15,38) {};
\node [v] at (10,33) {};
\node [v] at (28,20) {};
\node [v] at (23,25) {};
\node [v] at (42,20) {};
\node [v] at (47,25) {};
\node [v] at (35,13) {};
\node [v] at (35,8) {};

\node [v] at (48,45) {};
\node [v] at (43,50) {};
\node [v] at (62,45) {};
\node [v] at (67,50) {};
\node [v] at (55,38) {};
\node [v] at (60,33) {};

\node at (0,-1) {};

\end{tikzpicture}

\begin{tikzpicture}[scale=0.05]

\draw (8,45) -- (22,45) -- (15, 38) -- (8, 45);

\draw (28,20) -- (42,20) -- (35, 13) -- (28, 20);

\draw (48,45) -- (62,45) -- (55, 38) -- (48, 45);

\draw (8,45) -- (3, 50);
\draw (22,45) -- (27, 50);
\draw (15,38) -- (10, 33);
\draw (28,20) -- (23, 25);
\draw (42,20) -- (47, 25);
\draw (35,13) -- (35, 8);
\draw (48,45) -- (43, 50);
\draw (62,45) -- (67, 50);
\draw (55,38) -- (60, 33);

\draw (27,50) -- (47,25) -- (48,45) -- (27,50);

\node [v] at (8,45) {};
\node [v] at (3,50) {};
\node [v] at (22,45) {};
\node [v] at (27,50) {};
\node [v] at (15,38) {};
\node [v] at (10,33) {};
\node [v] at (28,20) {};
\node [v] at (23,25) {};
\node [v] at (42,20) {};
\node [v] at (47,25) {};
\node [v] at (35,13) {};
\node [v] at (35,8) {};

\node [v] at (48,45) {};
\node [v] at (43,50) {};
\node [v] at (62,45) {};
\node [v] at (67,50) {};
\node [v] at (55,38) {};
\node [v] at (60,33) {};

\end{tikzpicture}\quad
\begin{tikzpicture}[scale=0.05]

\draw (8,45) -- (22,45) -- (15, 38) -- (8, 45);

\draw (28,20) -- (42,20) -- (35, 13) -- (28, 20);

\draw (48,45) -- (62,45) -- (55, 38) -- (48, 45);

\draw (8,45) -- (3, 50);
\draw (22,45) -- (27, 50);
\draw (15,38) -- (10, 33);
\draw (28,20) -- (23, 25);
\draw (42,20) -- (47, 25);
\draw (35,13) -- (35, 8);
\draw (48,45) -- (43, 50);
\draw (62,45) -- (67, 50);
\draw (55,38) -- (60, 33);

\draw (27,50) -- (47,25) -- (43,50) -- (27,50);

\node [v] at (8,45) {};
\node [v] at (3,50) {};
\node [v] at (22,45) {};
\node [v] at (27,50) {};
\node [v] at (15,38) {};
\node [v] at (10,33) {};
\node [v] at (28,20) {};
\node [v] at (23,25) {};
\node [v] at (42,20) {};
\node [v] at (47,25) {};
\node [v] at (35,13) {};
\node [v] at (35,8) {};

\node [v] at (48,45) {};
\node [v] at (43,50) {};
\node [v] at (62,45) {};
\node [v] at (67,50) {};
\node [v] at (55,38) {};
\node [v] at (60,33) {};

\end{tikzpicture}

\caption{All non-isomorphic graphs in $\Delta_2$.}\label{fig:deltatwo}
\end{figure}

	All non-isomorphic graphs in $\Delta_2$ are depicted in Figure~\ref{fig:deltatwo}.
	Here is the main theorem of this section.

\begin{THM}\label{thm:main}
	Let $k$ be a non-negative integer. 
	Every graph in $\Delta_k$ is an excluded vertex-minor for graphs of linear rank-width at most $k$.
\end{THM}

\subsection{Linear rank-width of a graph in $\Delta_k$}

First, we prove that every graph in $\Delta_k$ has linear rank-width $k+1$.  

\begin{LEM}\label{lem:lowerbound}
  The linear rank-width of a graph in $\Delta_k$ is at least $k+1$.
\end{LEM}
\begin{proof}
	We use induction on $k$. 
	We may assume that $k\geq 1$. 
	Since $G\in \Delta_k$, $G$ is a delta composition of
        $G_1,G_2,G_3\in \Delta_{k-1}$
        with the main triangle 
	$v_1v_2v_3$ such that $v_i\in V(G_i)$ for $i=1$, $2$, $3$.

	Suppose that $G$ has linear rank-width at most $k$. 
        By the induction hypothesis, $G_1$ has linear rank-width at
        least $k$ and therefore $G$ has linear rank-width exactly $k$.
	Let $L$ be a linear layout of $G$ having width $k$. 
	For $v\in V(G)$, we define $S_v=\{x\in V(G):x\leq_L v\}$ and $T_v=V(G)\setminus S_v$. 
	Let $a$ and $b$ be the first and the last vertices in $L$ such that
	$\cutrk_G(S_a)=\cutrk_G(S_b)=k$. 
	Without loss of generality, we may assume that $\{a, b\}\subseteq V(G_2)\cup V(G_3)$. 
	Let $L_{1}$ be the subsequence of $L$ whose elements are the vertices of $G_1$.  

	For contradiction, we claim that $L_{1}$ is a linear layout of $G_1$ having width at most $k-1$. 
	Let $v\in V(G_1)$.
	It is sufficient to show that $\cutrk_{G_1}(S_v\cap V(G_1))\leq k-1$. 
	Note that $v\neq a$ and $v\neq b$. 
	If $v\le_L a$ or $v\ge_L b$, then 
        \[
        \cutrk_{G_1}(S_v\cap V(G_1))        \leq \cutrk_{G}(S_v)         \leq k-1.\]

	So we may assume that $a\le_L v\le_L b$. 
	Note that one of $S_v\cap V(G_1)$ and $T_v\cap V(G_1)$ does not have a neighbor in $G[V(G)\setminus V(G_1)]$ because $v_1$ is the unique vertex in $G_1$ which has a neighbor in $G[V(G)\setminus V(G_1)]$. 
	And since $G[V(G)\setminus V(G_1)]$ is connected and $a\in S_v\setminus V(G_1)$ and $b\in T_v\setminus V(G_1)$, 
	there is an edge $u_1u_2$ in $G[V(G)\setminus V(G_1)]$ such that $u_1\in S_v\setminus V(G_1)$ and $u_2\in T_v\setminus V(G_1)$. 
	So $A(G)[S_v\setminus V(G_1),T_v\setminus V(G_1)]$ is a non-zero matrix. 
	Depending on
 whether $v_1 \in S_v\cap V(G_1)$ or $v_1 \in T_v\cap V(G_1)$, 
\begin{multline*}
\cutrk_{G}(S_v)  \\
\,=\rank
\begin{pmatrix}
A(G)[S_v\cap V(G_1),T_v\cap V(G_1)]& 0 \\
A(G)[S_v\setminus V(G_1),T_v\cap V(G_1)]& A(G)[S_v\setminus V(G_1),T_v\setminus V(G_1)] \\  
\end{pmatrix} \\
\qquad\geq\rank\left(A(G)[S_v\cap V(G_1),T_v\cap V(G_1)]\right)+\rank\left(A(G)[S_v\setminus V(G_1),T_v\setminus V(G_1)]\right),
\end{multline*}
or  
\begin{multline*}
\cutrk_{G}(S_v) \\
\,=\rank
\begin{pmatrix}
A(G)[S_v\cap V(G_1),T_v\cap V(G_1)]& A(G)[S_v\cap V(G_1),T_v\setminus V(G_1)] \\
0& A(G)[S_v\setminus V(G_1),T_v\setminus V(G_1)] \\  
\end{pmatrix} \\
\qquad\geq\rank\left(A(G)[S_v\cap V(G_1),T_v\cap V(G_1)]\right)+\rank\left(A(G)[S_v\setminus V(G_1),T_v\setminus V(G_1)]\right), 
\end{multline*} respectively. Thus, we have  
\begin{align*}
\cutrk_{G_1}(S_v\cap V(G_1)	)
&=\rank\left(A(G)[S_v\cap V(G_1),T_v\cap V(G_1)]\right) \\
&\leq \cutrk_{G}(S_v)-\rank\left(A(G)[S_v\setminus V(G_1),T_v\setminus V(G_1)]\right) \\
&\leq \cutrk_{G}(S_v)-1 \leq k-1.
\end{align*} 
So $L_{1}$ is a linear layout of $G_1$ having width at most $k-1$, which is contradiction. Hence, $\lrw (G)\geq k+1$.
\end{proof}

   A vertex $w$ is called a \emph{twin} of another vertex $v$ in a graph
   if no vertex other than $v$ and $w$ is adjacent to exactly one of $v$ and $w$.

If $w$ is a twin of $v$ in a graph $G$
and 
$G\setminus w$ has linear rank-width $k+1$ with a linear layout
of width $k+1$ starting with $v$,
then clearly $G$ also admits a linear layout of width $k+1$ starting
with $v$
because we can easily put $w$ in the second place.
But the following lemma claims that we can place $w$ at
the end if $G\setminus w\in \Delta_{k}$.
This lemma implies that every graph in $\Delta_k$ has linear rank-width $k$.
Moreover, it will be mainly used in Section~\ref{sec:elementary}.

\begin{LEM}\label{lem:twin}
        Let $v$ be a vertex of a graph $G$
        and let $w$ be a twin of $v$.
	If $G\setminus w\in \Delta_k$,
        then 
	$G$ has a linear layout $L$  of width $k+1$ such that the first vertex of $L$ is $v$ and the last vertex of $L$ is $w$.  
\end{LEM}
Before proving the lemma, we first show that Lemma~\ref{lem:twin} 
implies the following proposition determining the exact linear
rank-width of a graph in $\Delta_k$.

\begin{PROP}\label{prop:lrwplus}
Every graph in $\Delta_k$ has linear
rank-width $k+1$.
Moreover, for every vertex $v$ of $G\in \Delta_k$,  there
exists a linear layout of $G$ having width $k+1$ whose first vertex is $v$.
\end{PROP}
\begin{proof}
  By Lemma~\ref{lem:lowerbound}, the linear rank-width of a graph $G$
  in $\Delta_k$ is at least $k+1$.
  Let $v\in V(G)$ and let $G'$ be a graph obtained by adding a twin $w$ of $v$ to $G$.
  Then Lemma~\ref{lem:twin} implies that $G'$ has a linear layout $L$
  of width $k+1$
  starting at $v$ and ending at $w$. 
  We discard $w$ from $L$ to obtain a linear layout of $G$
  starting with $v$ having width $k+1$.
\end{proof}
\begin{proof}[Proof of Lemma~\ref{lem:twin}]
	We prove by induction on $k$. 
        If $k=0$, then $G$ is a connected graph on three vertices
	and therefore every linear layout of $G$ has width $1$. 
	Thus we may assume that $k\geq 1$.
        Let $G\setminus w$ be a delta composition of $G_1,G_2,G_3\in
        \Delta_{k-1}$
        with the main triangle $v_1v_2v_3$ such that $v_i\in V(G_i)$ for $i=1$, $2$, $3$.
	We may assume that $v\in V(G_2)$. 
	
        We first claim that $G_1$ has a linear layout $L_1$ of width $k$ ending at $v_1$,
        and $G_3$ has a linear layout $L_3$ of width $k$ starting at $v_3$.
        For $i\in \{1,3\}$, let $G'_i$ is a graph obtained from $G_i$ by adding a twin $w_i$ of $v_i$.
        Since $G'_i\setminus w_i\in \Delta_{k-1}$, 
        by the induction hypothesis,
        $G'_i$ has a linear layout $L'_i$ of width $k$ starting at $w_i$ and ending at $v_i$,
        and by discarding $w_i$ from each $L'_i$, 
        we obtain a linear layout $L''_i$ of $G_i$ ending at $v_i$.
        So, $L_1=L''_1$ and the reverse layout $L_3$ of $L''_3$ are the linear layouts of $G_1$ and $G_3$ having width $k$, respectively, such that
        the last vertex of $L_1$ is $v_1$ and the first vertex of $L_3$ is $v_3$.

        Let 
        \[
        H=
        \begin{cases}
          G\setminus (V(G_1)\cup V(G_3)) 
          &\text{if }v\neq v_2,\\
          G\setminus (V(G_1)\cup V(G_3))\setminus vw &\text{if $v=v_2$, and $v$, $w$
        are adjacent in $G$,}\\
      G\setminus (V(G_1)\cup V(G_3))+vw &\text{otherwise.}
        \end{cases}
        \]
        By the induction hypothesis, 
        $H$ has a linear layout $(v)\oplus L_H\oplus (w)$ of width $k$.

        \smallskip\noindent (1) Clearly,
        $\cutrk_{G}(V(G_1)\cup\{v\})\le 2\le k+1$
        and 
        $\cutrk_{G}(V(G_3)\cup\{w\})\le 2\le k+1$.
	
        \smallskip\noindent (2) We claim that for $X\subseteq V(G_1)\setminus\{v_1\}$, if $\cutrk_{G_1}(X)\le
        k$, then $\cutrk_G(X\cup\{v\})\le k+1$.         
        This is because no vertex in $X$ has a neighbor in
        $V(G)\setminus V(G_1)$
        and therefore $\cutrk_{G_1}(X)=\cutrk_G(X)\ge
        \cutrk_G(X\cup\{v\})-1$ by the submodular inequality.

        \smallskip\noindent (3) 
        Similar to (2), we deduce that for $X\subseteq V(G_3)\setminus\{v_3\}$, if $\cutrk_{G_3}(X)\le
        k$, then $\cutrk_G(X\cup\{v\})\le k+1$. 
        
        \smallskip\noindent (4) 
        We claim that if $v\neq v_2$, $X\subseteq V(H)$, and    $\cutrk_{H}(X)\le k$, 
        then $\cutrk_G(V(G_1)\cup X)\le k+1$.
        By symmetry between $G_1$ and $G_3$, we may assume that
        $v_2\notin X$.
        By the submodular inequality, 
        $\cutrk_G(V(G_1)\cup X)\le
        \cutrk_G(X)+\cutrk_G(V(G_1))=\cutrk_H(X)+1\le k+1$.
        
        \smallskip\noindent (5) 
        We claim that if $v=v_2$, $v\in X\subseteq
        V(H) $, $w\notin X$, and         $\cutrk_{H}(X)\le k$, 
        then $\cutrk_G(V(G_1)\cup X)\le k+1$.
        By adding the row of $v_1$ to that of $v_2$ in $A(G)[X\cup
        V(G_1),(V(H)\setminus X)\cup V(G_3)]$, 
        we see that 
        $\cutrk_{G}(X\cup V(G_1))
        \le \cutrk_H(X)+1\le k+1$.

        By combining (1), (2), (3), (4), and (5),  we conclude that 
        $(v)\oplus L_1\oplus L_H\oplus L_3\oplus (w)$ is  a linear
        layout of $G$ having width at most $k+1$.
        Clearly it has width  $k+1$ because
        $G\setminus w$ has linear rank-width $k+1$ by
        Lemma~\ref{lem:lowerbound}.
\end{proof}

\subsection{Combining graphs in $\Delta_k$}
The following two lemmas will help us to prove that elementary
vertex-minors of  graphs in $\Delta_k$ have linear rank-width at most $k$.
\begin{LEM}\label{lem:join}
  Let $k$ be a positive integer 
  and let $G_1, G_2\in \Delta_{k-1}$.
  Let $G$ be a graph obtained from the disjoint union of $G_1$ and
  $G_2$ by adding an edge $w_1w_2$ for fixed $w_1\in V(G_1)$ and $w_2\in
  V(G_2)$.
  Then $G$ has linear rank-width $k$. 
\end{LEM}
\begin{proof}
  It is trivial that the linear rank-width of $G$ is at least $k$
  because  an induced subgraph $G_1$ of $G$
  has linear rank-width $k$ by Proposition~\ref{prop:lrwplus}.
  By Proposition \ref{prop:lrwplus}, 
  there is a linear layout $L_1$ of $G_1$ having width $k$ such that 
  the last vertex of $L_1$ is $w_1$, 
  and there is a linear layout $L_2$ of $G_2$ having width $k$ such that the first vertex of $L_2$ is $w_2$.
  Then obviously $L_1\oplus L_2$ is a linear layout of $G$ having
  width at most $k$. 
\end{proof}

\begin{LEM}\label{lem:twographs}
	Let $k$ be a positive integer. 
	Let $G_1$, $G_2\in \Delta_{k-1}$, and let $G_3$ be a graph having linear rank-width at most $k-1$. 
	Then every delta composition of $G_1$, $G_2$ and $G_3$ has linear rank-width $k$. 
\end{LEM}

\begin{proof}
	Let $G$ be a delta composition of $G_1$, $G_2$ and $G_3$
	with the main triangle $v_1v_2v_3$ such that $v_i\in V(G_i)$ for $i=1$, $2$, $3$. 
        Clearly the linear rank-width of $G$ is at least~$k$ because
        an induced subgraph $G_1$ of $G$ has linear rank-width
        $k$ by Proposition~\ref{prop:lrwplus}.

	Since $G_1$, $G_2\in \Delta_{k-1}$, by Proposition~\ref{prop:lrwplus}, there is a linear layout $L_1$ of $G_1$ having width $k$ such that 
	the last vertex of $L_1$ is $v_1$, 
	and there is a linear layout $L_2$ of $G_2$ having width $k$ such that the first vertex of $L_2$ is $v_2$. Let $L_3$ be a linear layout of $G_3$ having width at most $k-1$.

	We claim that $L=L_1\oplus L_3\oplus L_2$ is a linear layout of $G$ having width at most $k$. 
	Let $v\in V(G)$, $S_v=\{x:x\leq_L v\}$, and $T_v=V(G)\setminus S_v$. 
	We need to show that $\cutrk_G(S_v)\leq k$ for all $v\in V(G)$. 
	This is clearly true if $v\in V(G_1)\cup V(G_2)$. So let us
        assume that $v\in V(G_3)$.
        By symmetry we may assume $v_3\notin S_v$, because 
        we can swap $G_1$ and $G_2$.
        Then no vertex of $G_2$ has a neighbor in $S_v\cap V(G_3)$ and therefore
\begin{align*}
\cutrk_{G}(S_v)&\le 
\rank(A(G)[V(G_1),T_v])
+\rank(A(G)[S_v\cap V(G_3),T_v])\\
&= 1+ \cutrk_{G_3} (S_v\cap V(G_3))\le k.
\end{align*}
	Therefore, $G$ has linear rank-width at most $k$. 
   \end{proof}

\subsection{Linear rank-width of elementary vertex-minors of a graph in $\Delta_k$} \label{sec:elementary}

We will prove that every elementary vertex-minor of $G$ in $\Delta_k$
has linear rank-width at most $k$. To prove it, we will use the following lemmas.

\begin{LEM}[Bouchet \cite{Bouchet882}]\label{lem:bouchet}
Let $G$ be a graph, $v$ be a vertex of $G$ and $w$ be an arbitrary neighbor of $v$. Then every elementary vertex-minor obtained from $G$ by deleting $v$ is locally equivalent to either $G\setminus v$, $G* v\setminus v$, or $G\pivot vw\setminus v$.
\end{LEM}

\begin{LEM}[Bouchet {\cite[(8.2)]{Bouchet882}}; see Oum \cite{Oum05}]\label{lem:subsequentpivot}
Let $G$ be a graph and $vv_1, vv_2\in E(G)$. Then $v_1v_2\in E(G\pivot vv_1)$ and $G\pivot vv_1\pivot v_1v_2=G\pivot vv_2$. 
\end{LEM}

	By Lemma~\ref{lem:bouchet}, 
	it is sufficient to prove that $G\setminus v$, $G* v\setminus v$, and $G\pivot vw\setminus v$ 
	has linear rank-width one less than the linear rank-width of $G$.

\begin{LEM}\label{lem:elementarydel}
	Let $k$ be a non-negative integer and $G\in \Delta_k$. 
        Then $G\setminus v$ has linear rank-width at most $k$ for each
        vertex $v$.
\end{LEM}

\begin{proof}
	We use induction on $k$.
        We may assume $k\ge 1$.
	So 
	$G$ is a delta composition of three graphs in $\Delta_{k-1}$,
        say $G_1$, $G_2$ and $G_3$
        with the main triangle  $v_1v_2v_3$ such that $v_i\in V(G_i)$ for $i=1$, $2$, $3$. 
        We may assume that $v\in V(G_1)$. 
        By the induction hypothesis, $G_1\setminus v$ has linear
        rank-width at most $k-1$.

        If $v=v_1$, then $G\setminus v$ is obtained from the disjoint union of
        three graphs
        $G_1\setminus v$, $G_2$, $G_3$
        by adding an edge $v_2v_3$ and so $G\setminus v$ has linear
        rank-width $k$ by Lemma~\ref{lem:join}.

	If $v\neq v_1$, then 
        $G\setminus v$ is a delta composition of two graphs in $\Delta_{k-1}$ 
	and one graph having linear rank-width at most $k-1$. 
	Thus by Lemma~\ref{lem:twographs}, 
	$\lrw(G\setminus v)=k$.
\end{proof}

\begin{LEM}\label{lem:elementaryloc}
	Let $k$ be a non-negative integer and $G\in \Delta_k$. 
        Then $G*v\setminus v$ has linear rank-width at most $k$ for
        each vertex $v$.
\end{LEM}
\begin{proof}
	We use induction on $k$.
        We may assume $k\ge 1$.
	Let $G$ be a delta composition of $G_1,G_2,G_3\in
        \Delta_{k-1}$
        with the main triangle $v_1v_2v_3$ such that $v_i\in V(G_i)$ for $i=1$, $2$, $3$. 
        We may assume that $v\in V(G_1)$.

	If $v\neq v_1$, then 
        $G* v\setminus v$ is a delta composition of $G_1*v\setminus
        v$, $G_2$ and $G_3$ where $G_1*v\setminus v$ has linear
        rank-width at most $k-1$ by the induction hypothesis.
	Thus by Lemma~\ref{lem:twographs}, 
	$G* v\setminus v$ has linear rank-width $k$.

	So we may assume $v=v_1$. 
        let $G'_1=(G* v\setminus v)[V(G_1)\cup \{v_2, v_3\}]$.
        Since $v_3$ is a twin of $v_2$ in $G'_1$ and $v_3$ is not adjacent to $v_2$ in $G'_1*v_2$
        and $G'_1*v_2\setminus v_3$ is isomorphic to $G_1$ (see Figure~\ref{fig:sublocal}), 
        by Lemma~\ref{lem:twin},
        $G'_1$ has a linear layout $(v_2)\oplus L_1\oplus (v_3)$ of
        width $k$.

	By Proposition~\ref{prop:lrwplus}, $G_2$ has  a linear layout
        $L_{2}$ of width $k$ whose last vertex is $v_2$, 
	and $G_3$ has a linear layout $L_{3}$ of width $k$ whose first vertex is $v_3$. 

	It follows easily that $L=L_{2}\oplus L_{1}\oplus L_{3}$ 
        is a linear layout of $G* v\setminus v$ having width
        $k$
        because $(G*v\setminus v)[V(G_2)]=G_2$, $(G*v\setminus
        v)[V(G_3)]=G_3$, 
        and $(G*v\setminus v)[V(G_1)\cup\{v_2,v_3\}]=G'_1$.
\end{proof}

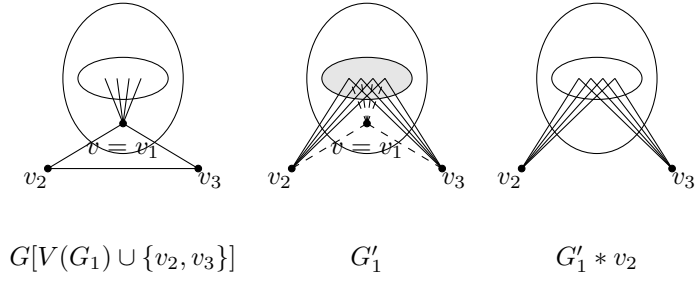
\begin{figure}\centering
 \tikzstyle{v}=[circle, draw, solid, fill=black, inner sep=0pt, minimum width=2.5pt]
\begin{tikzpicture}[scale=0.04]

\draw (40,50) -- (15,35) -- (65, 35) -- (40, 50);

\draw (40,65) ellipse (15 and 7);
\draw (40,50) -- (34,65);
\draw (40,50) -- (38,65);
\draw (40,50) -- (42,65);
\draw (40,50) -- (46,65);

\draw (40,65) ellipse (20 and 25);

\node at (11,31) {$v_2$};
\node at (69,31) {$v_3$};
\node at (40,41) {$v=v_1$};

\node [v] at (40,50) {};
\node [v] at (15,35) {};
\node [v] at (65,35) {};

\node at (40,5) {$G[V(G_1)\cup \{v_2,v_3\}]$};
	
\end{tikzpicture}\,
\begin{tikzpicture}[scale=0.04]

\draw[dashed] (40,50) -- (15,35);
\draw[dashed] (65, 35) -- (40, 50);

\draw[fill=gray!20] (40,65) ellipse (15 and 7);
\draw[dashed] (40,50) -- (34,65);
\draw[dashed] (40,50) -- (38,65);
\draw[dashed] (40,50) -- (42,65);
\draw[dashed] (40,50) -- (46,65);

\draw (15,35) -- (34,65);
\draw (15,35) -- (38,65);
\draw (15,35) -- (42,65);
\draw (15,35) -- (46,65);

\draw (65,35) -- (34,65);
\draw (65,35) -- (38,65);
\draw (65,35) -- (42,65);
\draw (65,35) -- (46,65);

\draw (40,65) ellipse (20 and 25);

\node at (11,31) {$v_2$};
\node at (69,31) {$v_3$};
\node at (40,41) {$v=v_1$};

\node [v] at (40,50) {};
\node [v] at (15,35) {};
\node [v] at (65,35) {};

\node at (40,5) {%
$G'_1$};
\end{tikzpicture}\,
\begin{tikzpicture}[scale=0.04]

\draw (40,65) ellipse (15 and 7);
\draw (15,35) -- (34,65);
\draw (15,35) -- (38,65);
\draw (15,35) -- (42,65);
\draw (15,35) -- (46,65);

\draw (65,35) -- (34,65);
\draw (65,35) -- (38,65);
\draw (65,35) -- (42,65);
\draw (65,35) -- (46,65);

\draw (40,65) ellipse (20 and 25);

\node at (11,31) {$v_2$};
\node at (69,31) {$v_3$};

\node [v] at (15,35) {};
\node [v] at (65,35) {};

\node at (40,5) {$G'_1* v_2$};

\end{tikzpicture}
\caption{The case $G* v\setminus v$ where $v=v_1$ in the proof of Lemma~\ref{lem:elementaryloc}.}
\label{fig:sublocal}
\end{figure}

\begin{LEM}\label{lem:elementarypiv}
	Let $k$ be a non-negative integer and $G\in \Delta_k$. 
        Then $G\pivot vw\setminus v$ has linear rank-width at most $k$
        for each edge $vw$.
\end{LEM}

\begin{proof}
  For each vertex $v$,  it is enough to prove it
  for one neighbor $w$  of $v$ 
  by Proposition~\ref{prop:leq} and Lemma~\ref{lem:subsequentpivot}.
  
	We use induction on $k$.
        We may assume $k\ge 1$.
	Let $G$ be a delta composition of $G_1,G_2,G_3\in \Delta_{k-1}$ 
	with the main triangle $v_1v_2v_3$ such that $v_i\in V(G_i)$ for $i=1$, $2$, $3$. 
        We may assume that $v\in V(G_1)$. 

        If $v$ has only one neighbor $w$, then $G\pivot vw\setminus v$
        is isomorphic to $G\setminus w$ and by
        Lemma~\ref{lem:elementarydel} we know that $G\setminus w$ has
        linear rank-width at most $k$.
        So we may assume that $v$ has at least two neighbors. 
	
	If $v\neq v_1$, 
        then we choose a neighbor $w$ of $v$ such that $w\neq v_1$.
	It is easy to observe that $G\pivot vw\setminus v$ is a  delta composition of 
        $G_1\pivot vw\setminus v$, $G_2$, $G_3$
        where
        $G_1\pivot vw\setminus v$ has linear rank-width at most $k-1$
        by the induction hypothesis. 
	Hence, by Lemma~\ref{lem:twographs}, $G\pivot vw\setminus v$
        has linear rank-width $k$. 

	Thus we may assume $v=v_1$. 
	Since $G[V(G_1)\cup \{v_2,v_3\}]\pivot vv_2\setminus v$ is
        isomorphic to a graph obtained from $G_1$ by adding a
        twin of $v$ (see Figure~\ref{fig:subpivot}),
	by Lemma~\ref{lem:twin}, 
        $G[V(G_1)\cup \{v_2,v_3\}]\pivot vv_2\setminus v$ has a linear
        layout $(v_2)\oplus L_1\oplus (v_3)$ 
        of width $k$.
        
	Let $w$ be a neighbor of $v$ in $G_1$
        and let $G'_1=G[V(G_1)\cup \{v_2,v_3\}]\pivot vw\setminus v$.
        By Lemma~\ref{lem:subsequentpivot}, 
        $G'_1\pivot v_2w = G[V(G_1)\cup \{v_2,v_3\}]\pivot vw \pivot
        v_2w \setminus v= G[V(G_1)\cup \{v_2,v_3\}]\pivot v v_2
        \setminus v$
        and therefore $(v_2)\oplus L_1\oplus (v_3)$ is also a linear
        layout of $G'_1$ having width $k$.

	By Proposition~\ref{prop:lrwplus}, $G_2$ has a linear layout
        $L_{2}$ of width $k$ whose last vertex  is $v_2$, 
	and $G_3$ has a linear layout $L_{3}$ of width $k$ whose first vertex is $v_3$. 
	
        It is now easy to see that 
        $L=L_{2}\oplus L_{1}\oplus L_{3}$ 
        is a linear layout of $G\pivot vw\setminus v$ having width at
        most $k$
        because $(G\pivot vw\setminus v)[V(G_2)]=G_2$,
        $(G\pivot vw\setminus v)[V(G_3)]=G_3$,
        and
        $(G\pivot vw\setminus v)[V(G_1)\cup\{v_2,v_3\}]=G'_1$.
\end{proof}

\begin{figure}\centering
 \tikzstyle{v}=[circle, draw, solid, fill=black, inner sep=0pt, minimum width=2.5pt]
\begin{tikzpicture}[scale=0.05]

\draw (40,50) -- (15,35) -- (65, 35) -- (40, 50);

\draw (40,65) ellipse (15 and 7);
\draw (40,50) -- (34,65);
\draw (40,50) -- (38,65);
\draw (40,50) -- (42,65);
\draw (40,50) -- (46,65);

\draw (40,65) ellipse (20 and 25);

\node at (11,31) {$v_2$};
\node at (69,31) {$v_3$};
\node at (40,41) {$v=v_1$};

\node [v] at (40,50) {};
\node [v] at (15,35) {};
\node [v] at (65,35) {};

\node at (40,15) {$G[V(G_1)\cup \{v_2,v_3\}]$};
	
\end{tikzpicture}\quad\quad
\begin{tikzpicture}[scale=0.05]

\draw[dashed] (40,50) -- (15,35) -- (65,35) ;
\draw (65, 35) -- (40, 50);

\draw (40,65) ellipse (15 and 7);
\draw (40,50) -- (34,65);
\draw (40,50) -- (38,65);
\draw (40,50) -- (42,65);
\draw (40,50) -- (46,65);

\draw (65,35) -- (34,65);
\draw (65,35) -- (38,65);
\draw (65,35) -- (42,65);
\draw (65,35) -- (46,65);

\draw (40,65) ellipse (20 and 25);

\node at (11,31) {$v$};
\node at (69,31) {$v_3$};
\node at (40,41) {$v_2$};

\node [v] at (40,50) {};
\node [v] at (15,35) {};
\node [v] at (65,35) {};

\node at (40,15) {$G[V(G_1)\cup \{v_2,v_3\}]\pivot v_1v_2\setminus v_1$};
	
\end{tikzpicture}
\caption{The case $G\pivot v_1v_2\setminus v$ in the proof of Lemma \ref{lem:elementarypiv}.}
\label{fig:subpivot}
\end{figure}
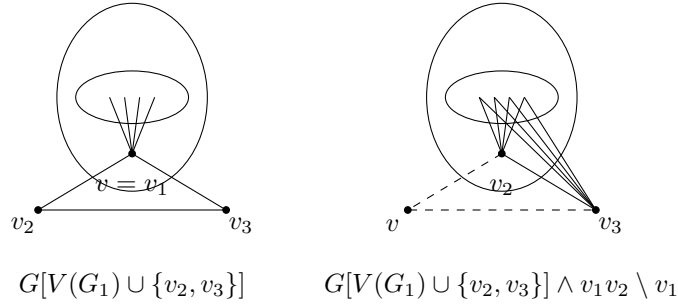

Finally we are ready to prove the main theorem of this section.
\begin{proof}[Proof of Theorem~\ref{thm:main}]
	Let $G\in \Delta_k$. 
	By Proposition~\ref{prop:lrwplus}, $G$ has linear rank-width $k+1$.
	And by
        lemmas~\ref{lem:elementarydel}, \ref{lem:elementaryloc}, and \ref{lem:elementarypiv}, 
        every elementary vertex-minor of $G$ has linear rank-width at
        most $k$. 
	Therefore, $G$ is an excluded vertex-minor for graphs of linear rank-width at most $k$.
\end{proof}

\section{Locally equivalent graphs in $\Delta_k$ are isomorphic}\label{sec:notlocal}

	In this section, we will prove that 
	if two graphs in $\Delta_k$ are locally equivalent, then they are isomorphic.
	We will prove the theorem for a more general class of graphs containing $\Delta_k$.
	
	A \emph{block} in a graph $G$ is a maximal connected subgraph of $G$
	having no cut-vertices.
    A graph is a \emph{block graph} if 
	every block of it is a complete graph.
        It is easy to see that every induced subgraph of a block graph
        is a block graph.
        
    A partition $(A,B)$ of $V(G)$ is a \emph{split} of a graph $G$ if
    $\abs{A}\geq 2$, $\abs{B}\geq 2$, and 
    $\cutrk_G(A)\le 1$.

	   We first show that every graph in $\Delta_k$
	 is a  block graph without simplicial vertices of degree at least $2$.
	\begin{LEM}\label{lem:csblock}
	Every graph in $\Delta_k$
	is a block graph without simplicial vertices of degree at
        least $2$.
	\end{LEM}
	\begin{proof}
          Let $G$ be a graph in $\Delta_k$. 
          From the construction of $\Delta_k$, 
          every vertex of $G$ has odd degree
          and each block of $G$ is isomorphic to $K_2$ or $K_3$.
          Therefore $G$ is a block graph
          and has no simplicial vertex of degree at least $2$. 
        \end{proof}

	We will prove that if two block graphs without simplicial
        vertices of degree at least $2$ are
        locally equivalent, then they are isomorphic.
	We will use the \emph{canonical decomposition} of a graph, 
	a useful tool introduced by Cunningham~\cite{Cunningham82}.

\subsection{Canonical decompositions of a connected graph}
	
	In this subsection, we will define the canonical decompositions of a connected graph,
        following the presentation by Bouchet~\cite{Bouchet88},
        and discuss the canonical decompositions of locally
        equivalent graphs.

	A \emph{marked graph} is a graph with a set of \emph{marked} edges, 
	and for a marked graph $D$, let $M(D)$ be the set of all marked edges of $D$. 
        A \emph{marked vertex} of a
        marked graph is a vertex incident with some marked edges.

	    A graph without splits is called \emph{prime}.
	If $G$ has a split $(A,B)$, then a marked graph $G'$ is 
	called a \emph{simple decomposition} of $G$ if 
        $G'$ is 
        obtained from the disjoint union of $G[A]$ and
        $G[B]$
        by 
        adding two new vertices $a$ and $b$, 
	adding a marked edge $ab$, %
	making $a$ adjacent to all vertices in $A$ having
        neighbors in $B$ in $G$, 
        and making $b$ adjacent to all vertices in $B$ having
        neighbors in $A$ in $G$.
		
	A \emph{split decomposition} of a connected graph $G$ is recursively defined to 
	be either $G$ with no marked edges 
        or a marked graph obtained from a split decomposition $D$ 
	by replacing a component $H$ of $D\setminus M(D)$ with a simple decomposition of $H$.

\begin{figure}\centering
\tikzstyle{v}=[circle, draw, solid, fill=black, inner sep=0pt, minimum width=2.5pt]
\tikzset{photon/.style={decorate, decoration={snake}}}
\begin{tikzpicture}[scale=0.04]

\path [fill=gray!20] (22,40) ellipse (21 and 15);
\path [fill=gray!20] (60,40) ellipse (10 and 15);
\path [fill=gray!20] (30,10) ellipse (12 and 10);

\draw (10, 50) -- (30, 50) -- (10,30) -- (30,30) -- (10, 50);
\draw (30,50) -- (40,40) -- (30,30);
\draw (55,40) -- (65,30) -- (65,50) -- (55,40);
\draw (30,15) -- (20,5) -- (40,5);

\draw[photon] (40,40) -- (55, 40);
\draw[photon] (30,30) -- (30, 15);

\draw (35,30) node{$c$};
\draw (35,15) node{$d$};
\draw (40,35) node{$e$};
\draw (55,35) node{$f$};

\node [v] at (10,50) {};
\node [v] at (30,50) {};
\node [v] at (65,50) {};
\node [v] at (10,30) {};
\node [v] at (30,30) {};
\node [v] at (65,30) {};

\node [v] at (40,40) {};
\node [v] at (55,40) {};

\node [v] at (30,15) {};
\node [v] at (20,5) {};
\node [v] at (40,5) {};

\draw (10,55) node{$v_1$};
\draw (30,55) node{$v_2$};
\draw (65,55) node{$v_3$};
\draw (10,25) node{$v_4$};
\draw (15,5) node{$v_5$};
\draw (65,25) node{$v_6$};
\draw (45,5) node{$v_7$};

\draw (95,-10) node{($\Longrightarrow$) Replacing a bag with a simple decomposition of it};

\draw (95,-20) node{($\Longleftarrow$) Recomposing along a marked edge $ab$};

\path [fill=gray!20] (30+90,40) ellipse (10 and 15);
\path [fill=gray!20] (30+120,40) ellipse (14 and 15);
\path [fill=gray!20] (60+120,40) ellipse (10 and 15);
\path [fill=gray!20] (30+120,10) ellipse (12 and 10);

\draw (-5+120 , 50) -- (5+120 ,40) -- (-5+120 ,30);
\draw (20+120 ,40) -- (30+120 ,50) -- (40+120 ,40) -- (30+120 ,30) -- (20+120 ,40);
\draw (55+120 ,40) -- (65+120 ,30) -- (65+120 ,50) -- (55+120 ,40);
\draw (30+120 ,15) -- (20+120 ,5) -- (40+120 ,5);

\draw[photon] (5+120 ,40) -- (20+120 , 40);
\draw[photon] (40+120 ,40) -- (55+120 , 40);
\draw[photon] (30+120 ,30) -- (30+120 , 15);

\draw (5+120 ,35) node{$a$};
\draw (20+120 ,35) node{$b$};
\draw (35+120 ,30) node{$c$};
\draw (35+120 ,15) node{$d$};
\draw (40+120 ,35) node{$e$};
\draw (55+120 ,35) node{$f$};

\node [v] at (-5+120 ,50) {};
\node [v] at (30+120 ,50) {};
\node [v] at (65+120 ,50) {};
\node [v] at (-5+120 ,30) {};
\node [v] at (30+120 ,30) {};
\node [v] at (65+120 ,30) {};

\node [v] at (5+120 ,40) {};
\node [v] at (20+120 ,40) {};
\node [v] at (40+120 ,40) {};
\node [v] at (55+120 ,40) {};

\node [v] at (30+120 ,15) {};
\node [v] at (20+120 ,5) {};
\node [v] at (40+120 ,5) {};

\draw (-5+120 ,55) node{$v_1$};
\draw (30+120 ,55) node{$v_2$};
\draw (65+120 ,55) node{$v_3$};
\draw (-5+120 ,25) node{$v_4$};
\draw (15+120 ,5) node{$v_5$};
\draw (65+120 ,25) node{$v_6$};
\draw (45+120 ,5) node{$v_7$};
\end{tikzpicture}
\caption{Two operations on a split decomposition.}\label{fig:recomposing}
\end{figure}

Let $D$ be a split decomposition of a connected graph $G$.
	Clearly $D$ is connected.
	Each component of $D\setminus M(D)$ is called a \emph{bag} of $D$.			If $ab$ is a marked edge in a split decomposition $D$,  
	then $D\pivot ab\setminus a\setminus b$ is called a split decomposition obtained by \emph{recomposing $ab$}. 
	See Figure~\ref{fig:recomposing} for an example.
	Given a split decomposition $D$, we can recover the graph $G$ 
        by recomposing all marked edges.
	Note that 
        the set of  vertices of $G$ is exactly the set of all
        unmarked vertices of $D$.
    It is easy to observe the following.
    \begin{LEM}\label{lem:induced}
    Let $D$ be a split decomposition of a connected graph $G$.
    If $B$ is a bag of $D$,
    then $G$ has an induced subgraph isomorphic to $B$.
    \end{LEM}
	
        A bag is called \emph{star} if it is isomorphic to $K_{1,n}$
        for some $n\ge 2$
        and it is called \emph{complete} if it is isomorphic to $K_{n}$
        for some $n\ge 1$.
        A non-leaf vertex of a star bag is called the \emph{center}.
	Two bags $C_1$ and $C_2$ of $D$ are \emph{neighbors} if there exist $v_1\in V(C_1)$, $v_2\in V(C_2)$ such that $v_1v_2\in M(D)$. 
	A split decomposition $D$ of a connected graph is called the \emph{canonical decomposition} if it satisfies the following: 

\begin{enumerate}[(i)]
	\item each bag of $D$ is prime, star, or complete,
	\item no two complete bags are neighbors,
	\item if two star bags are neighbors and $e$ is the marked edge connecting them, 
	then two end vertices of $e$ are both centers or both leaves of the bags. 
\end{enumerate}

The conditions (ii) and (iii) can be justified as follows. If there are
two complete bags that are neighbors, then we can recompose them to
create a bigger complete bag. If there are two star bags having a
marked edge joining a center of one to a leaf of another, we can also
recompose them to make a bigger star bag.
Thus the conditions (ii) and (iii) ensure that we do not decompose a complete or star
bag.
It turns out that each connected graph has a unique canonical decomposition.

\begin{LEM}[Cunningham \cite{Cunningham82}]\label{lem:canonicaleq}
  Every connected graph has a unique canonical decomposition.
\end{LEM}

In Appendix~\ref{sec:appendix}.
we present the canonical decompositions of graphs in $\Delta_k$.
    
A path in a marked graph is \emph{alternating}
if every second edge is marked and other edges are unmarked.
Let $D$ be a split decomposition of a connected graph $G$. 
    Two unmarked vertices $x$ and $y$ are \emph{linked} in $D$ if
    $D$ has an alternating path from $x$ to $y$. %
 The proof of the following lemma is an easy induction on the number of bags of the decomposition.
    
\begin{LEM}\label{lem:represent} 
Let $D$ be a split decomposition of a connected graph $G$ and let $v$, $w$ be two distinct unmarked vertices of $D$. 
 Then $v$ and $w$ are linked in $D$ if and only if $vw\in E(G)$.
\end{LEM}

	For $v\in V(G)$, we say a vertex $w$ in $D$ \emph{represents} $v$
	if %
	$D$ has an alternating path from $v$ to $w$ with even length
        (possibly $0$). %

	For $v\in V(G)$, let $D* v$ be a marked graph 
        obtained from $D$ by replacing $B$ with $B*w$
        for each bag $B$ of $D$ having a vertex $w$ representing $v$.
	See Figure~\ref{fig:loconsd} for an example.
	It is easy to observe that $D*v$ is a split decomposition of $G*v$.
        Moreover, the following lemma is known.  

\begin{figure}\centering
\tikzstyle{v}=[circle, draw, solid, fill=black, inner sep=0pt, minimum width=2.5pt]
\tikzset{photon/.style={decorate, decoration={snake}}}
\begin{tikzpicture}[scale=0.04]

\draw (10,50) -- (30,50) -- (50, 50) -- (50,30) -- (30,30) -- (10,30);
\draw (10,50) -- (30,30) -- (50, 50);
\draw (10,30) -- (30,50) -- (50,30);
\draw (30,30) -- (30,10);

\node [v] at (10,50) {};
\node [v] at (30,50) {};
\node [v] at (50,50) {};
\node [v] at (10,30) {};
\node [v] at (30,30) {};
\node [v] at (50,30) {};
\node [v] at (30,10) {};

\draw (10,55) node{$v_1$};
\draw (30,55) node{$v_2$};
\draw (50,55) node{$v_3$};
\draw (10,25) node{$v_4$};
\draw (25,25) node{$v_5$};
\draw (50,25) node{$v_6$};
\draw (30,5) node{$v_7$};

\draw (10,5) node{$G$};

\end{tikzpicture}\qquad\quad
\begin{tikzpicture}[scale=0.04]

\path [fill=gray!20] (0,40) ellipse (10 and 15);
\path [fill=gray!20] (30,40) ellipse (14 and 15);
\path [fill=gray!20] (60,40) ellipse (10 and 15);
\path [fill=gray!20] (30,10) ellipse (12 and 10);

\draw (-5, 50) -- (5,40) -- (-5,30);
\draw (20,40) -- (30,50) -- (40,40) -- (30,30) -- (20,40);
\draw (55,40) -- (65,30) -- (65,50) -- (55,40);
\draw (30,15) -- (20,5) -- (40,5);

\draw[photon] (5,40) -- (20, 40);
\draw[photon] (40,40) -- (55, 40);
\draw[photon] (30,30) -- (30, 15);

\draw (5,35) node{$a$};
\draw (20,35) node{$b$};
\draw (35,30) node{$c$};
\draw (35,15) node{$d$};
\draw (40,35) node{$e$};
\draw (55,35) node{$f$};

\node [v] at (-5,50) {};
\node [v] at (30,50) {};
\node [v] at (65,50) {};
\node [v] at (-5,30) {};
\node [v] at (30,30) {};
\node [v] at (65,30) {};

\node [v] at (5,40) {};
\node [v] at (20,40) {};
\node [v] at (40,40) {};
\node [v] at (55,40) {};

\node [v] at (30,15) {};
\node [v] at (20,5) {};
\node [v] at (40,5) {};

\draw (65,5) node{$D$};

\draw (-5,55) node{$v_1$};
\draw (30,55) node{$v_2$};
\draw (65,55) node{$v_3$};
\draw (-5,25) node{$v_4$};
\draw (15,5) node{$v_5$};
\draw (65,25) node{$v_6$};
\draw (45,5) node{$v_7$};

\end{tikzpicture}\qquad
\begin{tikzpicture}[scale=0.04]

\path [fill=gray!20] (0,40) ellipse (10 and 15);
\path [fill=gray!20] (30,40) ellipse (14 and 15);
\path [fill=gray!20] (60,40) ellipse (10 and 15);
\path [fill=gray!20] (30,10) ellipse (12 and 10);

\draw (-5, 50) -- (5,40) -- (-5,30) -- (-5,50);
\draw (20,40) -- (30,50) -- (40,40) -- (30,30) -- (20,40);
\draw (55,40) -- (65,30);
\draw (65,50) -- (55,40);
\draw (30,15) -- (20,5) -- (40,5);
\draw (20,40) -- (40,40);

\draw[photon] (5,40) -- (20, 40);
\draw[photon] (40,40) -- (55, 40);
\draw[photon] (30,30) -- (30, 15);

\draw (5,35) node{$a$};
\draw (20,35) node{$b$};
\draw (35,30) node{$c$};
\draw (35,15) node{$d$};
\draw (40,35) node{$e$};
\draw (55,35) node{$f$};

\node [v] at (-5,50) {};
\node [v] at (30,50) {};
\node [v] at (65,50) {};
\node [v] at (-5,30) {};
\node [v] at (30,30) {};
\node [v] at (65,30) {};

\node [v] at (5,40) {};
\node [v] at (20,40) {};
\node [v] at (40,40) {};
\node [v] at (55,40) {};

\node [v] at (30,15) {};
\node [v] at (20,5) {};
\node [v] at (40,5) {};

\draw (65,5) node{$D* v_2$};

\draw (-5,55) node{$v_1$};
\draw (30,55) node{$v_2$};
\draw (65,55) node{$v_3$};
\draw (-5,25) node{$v_4$};
\draw (15,5) node{$v_5$};
\draw (65,25) node{$v_6$};
\draw (45,5) node{$v_7$};

\end{tikzpicture}
\caption{From a split decomposition $D$ of a graph $G$, we obtain a
  split decomposition $D* v_2$ of $G*v_2$. Note
  that $v_2$ is represented by
  $v_2,a,f$ in $D$.}\label{fig:loconsd}
\end{figure}
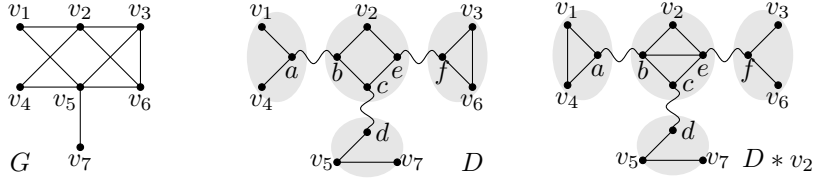 

\begin{LEM}[Bouchet \cite{Bouchet88}]\label{lem:canonicallocal}
	If $D$ is the canonical decomposition of a connected graph $G$ and $v\in V(G)$, 
	then $D* v$ is the canonical decomposition of $G* v$. 
\end{LEM}

A graph $G$ is \emph{distance-hereditary}~\cite{BM1986}
if for each connected induced subgraph $H$ of $G$ and two distinct vertices $x,y$ in $H$, their distance in $H$ is the same as in $G$. 
It is known that connected distance-hereditary graphs are exactly the graphs
having the canonical decomposition whose bags are either star or
complete~\cite{Bouchet88}. 
It is easy to see that every block graph is distance-hereditary~\cite{BM1986}.

\subsection{Canonical decompositions of block graphs}

	A \emph{diamond graph} is the graph obtained from $K_4$ by
        removing one edge.
        By definition, neither a diamond graph nor $C_k$ for
        $k\ge 4$ is a
        block graph.
        Actually 
        Bandelt and Murder~\cite{BM1986}
        showed that a graph is a block graph if and only if it has no
        induced subgraph isomorphic to a diamond  graph or $C_k$ for
        $k\ge 4$. 
        In the following proposition, we will characterize block graphs from their canonical
        decompositions.
	\begin{PROP}\label{prop:charblock2}
          Let $D$ be the canonical decomposition of a connected
          graph $G$. 
	Then $G$ is a block graph if and only if 
        every bag of $D$ is either star or complete,
        and 
	the center of each star bag of $D$  is unmarked.
	\end{PROP}	
	\begin{proof}
          We may assume that $G$ is distance-hereditary because
          otherwise $D$ has a bag that is neither star nor complete,
          and $G$ is not a block graph.

	We first suppose that 
        $D$ has a star bag $B$ having a marked center $w$.
        There exists a marked edge $ww'$ joining $B$ with a
        bag $B'$.
	Since $D$ is a canonical decomposition,
	$B'$ is 
        either complete or 
        star with the center $w'$.
	If $B'$ is complete,
	then by recomposing $ww'$ we obtain a bag which has an induced
        subgraph isomorphic to a diamond graph.
	Thus $G$ has an induced subgraph isomorphic to a diamond graph
        by Lemma~\ref{lem:induced}. Since a diamond graph is not a block
        graph, we deduce that $G$ is not a block graph.
	If $B'$ is a star bag with the center $w'$,
	then by recomposing $ww'$, we obtain a bag which has an
        induced subgraph isomorphic to $C_4$. 
	By Lemma~\ref{lem:induced}, $G$ should have an induced
        subgraph isomorphic to $C_4$, and therefore $G$ is not a block
        graph.
        
        To prove the converse, we claim a stronger statement:
        if $D$ is a \emph{split} decomposition  of a connected graph $G$ 
        whose bags are star or complete and no center of a star bag in
        $D$ is marked, then $G$ is a block graph.
        We proceed by induction on $\abs{V(D)}$.
        We may assume that $D$ has a star bag $B$ because otherwise
        $G$ is a complete graph. 
        Let $v$ be the center of $B$.
        If $B$ has another unmarked vertex $w$, then let $G'$ be a
        graph obtained  by recomposing all marked  edges in
        $D\setminus w$. Here $G$ is obtained from $G'$ by adding a
        pendant vertex $w$ to $v$. By the induction hypothesis, $G'$
        is a block graph and so is $G$.
        We may now assume that every vertex in $B$ other than $v$ is
        marked.
        Let $B=\{v, v_1,v_2,\ldots,v_n\}$ and let
        $v_1w_1,v_2w_2,\ldots,v_nw_n$ be the marked edges incident with
        $B$.
        Let $D_i$ be the component of $D\setminus V(B)$ containing
        $w_i$.
        By the induction hypothesis, 
        the graph $G_i$ obtained by recomposing all marked edges in
        $D_i$ is a block graph.
        The graph $G$ is obtained from $G_1,G_2,\ldots,G_n$ by identifying
        $w_1,w_2,\ldots,w_n$ with a new vertex  $v$. Since each block of $G$ is a
        block of $G_i$ for some $i$, we deduce that $G$ is a block graph.
	\end{proof}

        We now characterize block graphs without simplicial vertices
        of degree at least $2$ in terms of their canonical
        decompositions. 
        \begin{PROP}\label{prop:cliquesplit}
          Let $D$ be the canonical decomposition of a connected block
          graph $G$.
          Then $G$ has a simplicial vertex of degree at least $2$
          if and only if 
          $D$ has a complete bag $B$ having more than $2$ vertices
          containing an unmarked vertex.
        \end{PROP}
        \begin{proof}
          Suppose that $v\in V(G)$ is a simplicial vertex of degree at least
          $2$ in $G$.
          Clearly $v$ is not a center of a star bag of $D$ by
          Lemma~\ref{lem:induced}.
          Because the center of a star bag is unmarked by
          Proposition~\ref{prop:charblock2}
          and $v$ has degree at least $2$, 
          $v$ cannot belong to a star bag.
          So $v$ is in a complete bag of $D$.

          Conversely suppose that $D$ has a complete bag $B$ having
          more than $2$ vertices containing an unmarked vertex $v$.
          By Lemma~\ref{lem:induced}, 
          the degree of $v$ is at least $2$.
          Since all neighboring bags of $B$ are star bags whose centers
          are unmarked by Proposition~\ref{prop:charblock2}, 
          $v$ is a simplicial vertex of $G$.
        \end{proof}

\subsection{Generalizing a theorem of Bouchet}

	Now we are ready to prove Theorem~\ref{thm:blockbou}.
        This theorem is best possible for block graphs, because
        if $v$ is a simplicial vertex of a block graph $G$, 
        then $G*v$  is also a block graph.
	\begin{thmext}
	If two block graphs without simplicial vertices of degree at
        least $2$ are locally equivalent, %
	then they are isomorphic. 
	\end{thmext}

	\begin{proof}%
	Suppose that two block graphs $G$ and $H$ without simplicial
        vertices of degree at least $2$ are locally equivalent. %
	Let $D_G$ and $D_H$ be the canonical decompositions of $G$ and
        $H$, respectively.
	We may assume that $\abs{V(G)}=\abs{V(H)}\ge 3$
        and therefore each bag  of $D_G$ or $D_H$ has at least $3$ vertices.

	Since $G$ and $H$ are locally equivalent, 
	by Lemma~\ref{lem:canonicallocal} 
	we assume that $D_H$ is obtained from $D_G$ by a sequence of local complementations.
	Note that applying local complementation in a split decomposition does not change the number of marked vertices and unmarked vertices in each bag.
	
        Suppose that a bag $B$ of $D_G$ corresponds to a bag
        $B'=D_H[V(B)]$ of $D_H$.
        If $B$ is a complete bag in $D_G$, then
        by Proposition~\ref{prop:cliquesplit},
        $B$ has no unmarked vertex in $D_G$ and therefore
        $B'$ has no unmarked vertex in $D_H$.
        Since every star bag of $D_H$ should have at least one unmarked vertex by Proposition~\ref{prop:charblock2},
	$B'$ is a complete bag in $D_H$. 
	Similarly, if $B'$ is a complete bag in $D_H$,  then $B$ is a
        complete bag in $D_G$. 

        Thus $B$ is a star bag of $D_G$ if and only if $B'$ is a star bag of $D_H$.
        By Proposition~\ref{prop:charblock2},
	the center of a star bag in $D_G$ or $D_H$
        is an unmarked vertex.
	Since a bag $B$ in $D_G$ and $B'$ in $D_H$
have the same number of neighbor bags and unmarked vertices in each canonical decomposition, 
	the unmarked vertices of $B$ in $D_G$ must be mapped to the
        unmarked vertices of $B'$ in
        $D_H$.
	Therefore, $D_G$ is isomorphic to $D_H$ and so $G$ is isomorphic to $H$.
	\end{proof}

\section{The number of non-isomorphic graphs in $\Delta_k$} \label{sec:size}

In this section, we will prove that 
$\Delta_k$ 
has at least $2^{\Omega(3^k)}$ pairwise non-isomorphic graphs.
A \emph{rooted graph} is a  pair of a graph and a specified vertex called a
\emph{root}.
Two rooted graphs $(G,v)$ and $(G',v')$ are \emph{isomorphic}
if 
there exists a graph
isomorphism $\phi$ from $G$ to $G'$ that maps $v$ to $v'$.
Let us write $\Aut(G)$ to denote the automorphism group of a graph
$G$.
For a rooted graph $(G,v)$, we write $\Aut(G,v)$ to 
denote the automorphism group of $(G,v)$.
In other words,  
$\Aut(G,v)=\{\phi\in \Aut(G): \phi(v)=v\}$.

First we show that each graph in $\Delta_k$ has a unique main triangle.
\begin{LEM}\label{lem:maintri}
  Let $k\ge 1$ and $G\in \Delta_k$.
  Then $G$ has a unique cycle $v_1v_2v_3$ of length $3$
  such that $G\setminus v_1v_2\setminus v_2v_3\setminus v_3v_1$ has
  exactly three components $G_1,G_2,G_3$, each of which is in $\Delta_{k-1}$.
\end{LEM}
\begin{proof}
  Clearly there is at least one such cycle because of the
  construction. Suppose there are two such cycles $T=v_1v_2v_3$ and
  $T'=v_1'v_2'v_3'$.
  Let $H$ be a component of $G\setminus v_1v_2\setminus
  v_2v_3\setminus v_3v_1$ having no vertex of $T'$. By the
  condition, $H\in \Delta_{k-1}$ and so $H$ has exactly $2\cdot
  3^{k-1}$ vertices. We may assume $v_1\in V(H)$.
  The component $J$ of $G\setminus v_1'v_2'\setminus
  v_2'v_3'\setminus v_3'v_1'$ intersecting $V(H)$ should be equal to
  $H$ because $T'$ does not intersect $H$ and
  $\abs{V(J)}=\abs{V(H)}$. Thus $v_2,v_3\in T'$ and so $v_2$ and $v_3$
  have a common neighbor other than $v_1$. However, this contradicts
  our assumption that $G\setminus v_1v_2\setminus v_2 v_3\setminus
  v_3v_1$ has exactly three components.
\end{proof}

Let $k\ge 2$ and let $G$ be a graph in $\Delta_k$.
By the construction, $G$ is a delta composition of three graphs $G_1,
G_2,G_3\in \Delta_{k-1}$ with the main triangle $v_1v_2v_3$
such that $v_i\in V(G_i)$ for $i=1,2,3$.
We call $G\in \Delta_k$ 
\begin{itemize}
\item \emph{Type-A} if $(G_1,v_1)$, $(G_2,v_2)$, and $(G_3,v_3)$
are pairwise isomorphic,
\item \emph{Type-B} if exactly two of $(G_1,v_1)$, $(G_2,v_2)$,
  $(G_3,v_3)$ are isomorphic, 
\item \emph{Type-C} otherwise.
\end{itemize}

\begin{LEM}\label{lem:fixauto}
  Let $k\ge 1$ and $G$ be a delta composition of three graphs
  $G_1,G_2,G_3\in \Delta_{k-1}$ with the main triangle $v_1v_2v_3$
  such that $v_i\in V(G_i)$ for all $i=1,2,3$. Then,
  \begin{enumerate}
  \item $\Aut(G)\simeq S_3\times \Aut(G_1,v_1)\times \Aut(G_2,v_2)\times
    \Aut(G_3,v_3)$
    if $G$ is Type-A.
  \item $\Aut(G)\simeq S_2\times \Aut(G_1,v_1)\times \Aut(G_2,v_2)\times
    \Aut(G_3,v_3)$
    if $G$ is Type-B.
  \item $\Aut(G)\simeq \Aut(G_1,v_1)\times \Aut(G_2,v_2)\times
    \Aut(G_3,v_3)$
    if $G$ is Type-C.
  \end{enumerate}
\end{LEM}
\begin{proof}
  Let $g\in\Aut(G)$.
  By Lemma~\ref{lem:maintri}, $g(\{v_1,v_2,v_3\})=\{v_1,v_2,v_3\}$
  and therefore $g(V(G_1)), g(V(G_2)), g(V(G_3))\in\{V(G_1),V(G_2),V(G_3)\}$.
  So $\Aut(G)$ induces a subgroup $\Gamma$ of $S_3$ on $\{v_1,v_2,v_3\}$ based
  on the type of $G$. It is clear that $\Aut(G)/\Gamma$ is a
  composition of automorphism groups of three rooted graphs $(G_1,v_1)$, $(G_2,v_2)$ and $(G_3,v_3)$.
\end{proof}

	For a graph $G$ and $x\in V(G)$, we define the \emph{orbit} of $x$ in $G$ as the set
	\[
	\{ w\in V(G): w=f(x) \text{ for some automorphism $f$ of $G$}\},
	\]
	and we denote $\norb(G)$ as the number of all distinct orbits of $G$.
	For a rooted graph $(G,v)$ and $x\in V(G)$, we define the \emph{orbit} of $x$ in $(G,v)$ as the set
	\[
	\{ w\in V(G): w=f(x) \text{ for some automorphism $f$ of $(G,v)$}\},
	\]
	and we denote $\norb(G,v)$ as the number of all distinct orbits of $(G,v)$.

\begin{LEM}\label{lem:orbitineq}
  Let $k\ge 1$ and $G$ be a delta composition of three graphs
  $G_1,G_2,G_3\in \Delta_{k-1}$
  with the main triangle $v_1v_2v_3$ such that $v_i\in V(G_i)$ for all
  $i=1,2,3$. If $v\in V(G_1)$, 
  then 
  \[\norb(G,v)\ge \norb(G_1,v_1)+\norb(G_2,v_2).
  \]
\end{LEM}
\begin{proof}
  By Lemma~\ref{lem:maintri}, no vertex in $G_1$ can be mapped to a
  vertex in $G_2$ or $G_3$ by an automorphism of $G$ fixing $v$.
  Thus orbits of $(G,v)$ intersecting $V(G_1)$ cannot contain a
  vertex in $G_2$ or $G_3$.
    The number of orbits of $(G,v)$ intersecting $V(G_1)$ is equal
  to %
  the number of distinct subsets of $V(G_1)$ that can be represented as 
  \[
  \{f(x)\in V(G_1)\colon \text{$f$ is an automorphism of $G_1$
    such that $f(v)=v$, $f(v_1)=v_1$}\}
  \]
  for some $x\in V(G_1)$
  and this number is at least $\norb(G_1,v_1)$.
  The number of orbits of $(G,v)$ not intersecting $V(G_1)$ is at
  least $\norb(G_2,v_2)$ by Lemma~\ref{lem:fixauto}.
  Thus, we obtain the desired inequality.
\end{proof}
\begin{LEM}\label{lem:orbitineq2}
  Let $k$ be a non-negative integer and $G\in \Delta_k$ and $v\in V(G)$. 
  Then $(G,v)$ has at least $2^{k+1}$ orbits.
\end{LEM}
\begin{proof}
  Trivial if $k=0$. It follows easily by induction from Lemma~\ref{lem:orbitineq}.
\end{proof}

\begin{LEM}\label{lem:orbit}
Let $k$ be a positive integer and $G\in \Delta_k$.
\begin{enumerate}
\item If $G$ is Type-A, then $G$ has at
  least $2^k$ orbits.
\item If $G$ is Type-B, then $G$ has at
  least $2\cdot 2^k$ orbits.
\item If $G$ is Type-C, then $G$ has at
  least $3\cdot 2^k$ orbits.
\end{enumerate}
\end{LEM}
\begin{proof}
  Let $G$ be a delta composition of $G_1, G_2, G_3\in\Delta_{k-1}$
  with the main triangle $v_1v_2v_3$ such that $v_i\in V(G_i)$ for all $i=1,2,3$.
  By Lemma~\ref{lem:fixauto}, 
  \begin{enumerate}
  \item   $\norb(G)=\norb(G_1,v_1)$ if $G$ is Type-A,
  \item   $\norb(G)=\norb(G_1,v_1)+\norb(G_2,v_2)$ if $G$ is Type-B
    and $(G_1,v_1)$ is isomorphic to $(G_3,v_3)$,
  \item 
    $\norb(G)=\norb(G_1,v_1)+\norb(G_2,v_2)+\norb(G_3
    ,\linebreak[0]v_3)$ if $G$ is Type-C.
  \end{enumerate}
  By Lemma~\ref{lem:orbitineq2}, we deduce the lemma.
\end{proof}

Let $p_k$ be the number of non-isomorphic rooted graphs $(G,v)$ with $G\in\Delta_k$.
Then $p_0=1$, $p_1=2$, and $p_2=24$ (see Figure~\ref{fig:deltatwo}). 
We can easily verify that $\Delta_k$ has 
\begin{itemize}
\item exactly $p_{k-1}$ non-isomorphic Type-A graphs,
\item  exactly $p_{k-1}(p_{k-1}-1)$ non-isomorphic Type-B  graphs, 
\item  exactly $\binom{p_{k-1}}{3}$ non-isomorphic Type-C graphs.
\end{itemize}

We are now ready to provide a lower bound on the number of
non-isomorphic graphs in $\Delta_k$.

\begin{PROP}\label{prop:sizedelta}
Let $k \ge 2$ be an integer. Then $\Delta_k$ has at least
$2^{\Omega(3^k)}$ non-isomorphic graphs.
\end{PROP}

\begin{proof}
Let $a_k$, $b_k$, $c_k$ be the number of non-isomorphic graphs in
$\Delta_k$ that is Type-A, Type-B, and Type-C respectively.
By Lemma~\ref{lem:orbit}, 
\[
p_k\ge 2^k a_k + 2 \cdot 2^k b_k + 3 \cdot 2^k c_k.
\]
Since $a_k=p_{k-1}$, $b_k=p_{k-1}(p_{k-1}-1)$ and $c_k=\binom{p_{k-1}}{3}$, we obtain the following recurrence relation;
\[
a_{k+1}
=p_k \ge 2^k {a_k} + 2 \cdot 2^k {b_k} + 3 \cdot 2^k {c_k} 
=2^{k-1} a_k^2 (a_k+1)
\ge 2^{k-1}a_k^3
\]
and ${a_2}=2$. 
We deduce that $a_k\ge 2^{(1-2k)/4+7\cdot 3^k/36}= 2^{\Omega(3^k)}$. 
\end{proof}

Now we can combine all to prove our main theorem.
\begin{proof}[Proof of Theorem~\ref{thm:bigmain}]
By Theorems~\ref{thm:main}, 
$\mathcal O_k$ must contain a graph locally equivalent to each graph
in $\Delta_k$.
Proposition~\ref{prop:sizedelta} states that $\Delta_k$ has at least
$2^{\Omega(3^k)}$ non-isomorphic graphs.
Lemma~\ref{lem:csblock} and Theorem~\ref{thm:blockbou}
 show that two non-isomorphic
graphs in $\Delta_k$ cannot be locally equivalent. 
Therefore, $\abs{\mathcal{O}_k}\geq 2^{\Omega(3^k)}$.
\end{proof}

\section{Concluding remarks}\label{sec:final}

	We present $2^{\Omega(3^k)}$ lower bound of 
	the number of pairwise locally non-equivalent vertex-minor minimal graphs 
	with the property that they have linear rank-width larger than $k$.

    A question naturally arises in the context.
\begin{QUE}  
  Find an explicit upper bound on the number of vertices in 
    a graph that is vertex-minor minimal 
    with the property having  linear rank-width larger than $k$.
\end{QUE}
So far, we do not know any explicit upper bound; its existence is
given by Corollary~\ref{cor:vertexminorwqo2}. 
The only known fixed-parameter algorithm to decide linear rank-width
at most $k$ is based on this list; it uses the modulo-$2$ counting
monadic second-order logic formula to decide whether a given graph has
linear rank-width at most $k$ by using the existence of forbidden
vertex-minors. 
However, no explicit methods are known to construct such a list of
forbidden vertex-minors and so perhaps we can say ``we know such an
algorithm exists but we do not know what it is.''

A similar problem occurs in the problem of
deciding rank-width at most $k$ in Courcelle and Oum~\cite{CO2007}.
But for rank-width, there is an explicit upper bound on the number of
vertices of forbidden vertex-minors \cite{Oum05}
and therefore in theory, one can enumerate all graphs up to that bound
and construct the list of forbidden vertex-minors.
If we resolve the above question, then we will be able to \emph{construct} a
fixed parameter algorithm to decide linear rank-width at most $k$.

%
%
%
%

%
%
%
%


\begin{thebibliography}{10}

\bibitem{AFP2013}
Isolde Adler, Arthur~M. Farley, and Andrzej Proskurowski., \emph{Obstructions
  for linear rankwidth at most $1$}, Discrete Appl. Math. \textbf{168} (2014),
  3--13.

\bibitem{BM1986}
Hans-J{\"u}rgen Bandelt and Henry~Martyn Mulder, \emph{Distance-hereditary
  graphs}, J. Combin. Theory Ser. B \textbf{41} (1986), no.~2, 182--208.

\bibitem{Bouchet882}
Andr{\'e} Bouchet, \emph{Graphic presentations of isotropic systems}, J.
  Combin. Theory Ser. B \textbf{45} (1988), no.~1, 58--76.

\bibitem{Bouchet88}
\bysame, \emph{Transforming trees by successive local complementations}, J.
  Graph Theory \textbf{12} (1988), no.~2, 195--207.

\bibitem{Bouchet1989a}
\bysame, \emph{Connectivity of isotropic systems}, Combinatorial Mathematics:
  Proceedings of the Third International Conference (New York, 1985) (New
  York), Ann. New York Acad. Sci., vol. 555, New York Acad. Sci., 1989,
  pp.~81--93. \MR{90i:05023}

\bibitem{CO2007}
Bruno Courcelle and Sang-il Oum, \emph{Vertex-minors, monadic second-order
  logic, and a conjecture by {S}eese}, J. Combin. Theory Ser. B \textbf{97}
  (2007), no.~1, 91--126. \MR{2278126 (2007i:68093)}

\bibitem{Cunningham82}
William~H. Cunningham, \emph{Decomposition of directed graphs}, SIAM J.
  Algebraic Discrete Methods \textbf{3} (1982), no.~2, 214--228.

\bibitem{DAT12}
Zden{\v{e}}k Dvo{\v{r}}{\'a}k, Archontia~C. Giannopoulou, and Dimitrios~M.
  Thilikos, \emph{Forbidden graphs for tree-depth}, European J. Combin.
  \textbf{33} (2012), no.~5, 969--979.

\bibitem{Ganian10}
Robert Ganian, \emph{Thread graphs, linear rank-width and their algorithmic
  applications}, Combinatorial Algorithms, Lecture Notes in Comput. Sci., vol.
  6460, Springer, 2011, pp.~38--42.

\bibitem{GGRW03}
J.~F. Geelen, A.~M.~H. Gerards, N.~Robertson, and G.~P. Whittle, \emph{On the
  excluded minors for the matroids of branch-width {$k$}}, J. Combin. Theory
  Ser. B \textbf{88} (2003), no.~2, 261--265.

\bibitem{GKS99}
Arvind Gupta, Damon Kaller, and Thomas Shermer, \emph{On the complements of
  partial {$k$}-trees}, Automata, languages and programming ({P}rague, 1999),
  Lecture Notes in Comput. Sci., vol. 1644, Springer, Berlin, 1999,
  pp.~382--391.

\bibitem{JKO2013}
Jisu Jeong, O-joung Kwon, and Sang{-}il Oum, \emph{Excluded vertex-minors for
  graphs of linear rank-width at most $k$}, 30th International Symposium on
  Theoretical Aspects of Computer Science (STACS 2013) (Kiel, Germany) (Natacha
  Portier and Thomas Wilke, eds.), Leibniz International Proceedings in
  Informatics (LIPIcs), vol.~20, Schloss Dagstuhl. Leibniz-Zent. Inform., 2013,
  pp.~221--232.

\bibitem{LAG98}
Jens Lagergren, \emph{Upper bounds on the size of obstructions and
  intertwines}, J. Combin. Theory Ser. B \textbf{73} (1998), no.~1, 7--40.
  \MR{1620599 (99d:05047)}

\bibitem{Oum05}
Sang{-}il Oum, \emph{Rank-width and vertex-minors}, J. Combin. Theory Ser. B
  \textbf{95} (2005), no.~1, 79--100.

\bibitem{Oum082}
\bysame, \emph{Rank-width and well-quasi-ordering}, SIAM J. Discrete Math.
  \textbf{22} (2008), no.~2, 666--682.

\bibitem{OS2004}
Sang{-}il Oum and Paul Seymour, \emph{Approximating clique-width and
  branch-width}, J. Combin. Theory Ser. B \textbf{96} (2006), no.~4, 514--528.

\bibitem{TASY94}
Atsushi Takahashi, Shuichi Ueno, and Yoji Kajitani, \emph{Minimal acyclic
  forbidden minors for the family of graphs with bounded path-width}, Discrete
  Math. \textbf{127} (1994), no.~1-3, 293--304, Graph theory and applications
  (Hakone, 1990).

\bibitem{Thilikos00}
Dimitrios~M. Thilikos, \emph{Algorithms and obstructions for linear-width and
  related search parameters}, Discrete Appl. Math. \textbf{105} (2000),
  no.~1-3, 239--271.

\end{thebibliography}
\providecommand{\bysame}{\leavevmode\hbox to3em{\hrulefill}\thinspace}
\providecommand{\MR}{\relax\ifhmode\unskip\space\fi MR }
\providecommand{\MRhref}[2]{%
  \href{http://www.ams.org/mathscinet-getitem?mr=#1}{#2}
}
\providecommand{\href}[2]{#2}

\begin{appendix}
\section{Canonical decompositions of graphs in $\Delta_k$}\label{sec:appendix}
        \renewcommand{\theTHM}{\Alph{section}.\arabic{THM}}

	We now aim to describe the canonical decomposition $D_G$ of
        each graph $G$ in $\Delta_k$ for $k\ge 1$ explicitly.
        Let us call the edges of the graph in $\Delta_0$ \emph{thick}.
        In graphs in $\Delta_k$, the edges originated from $\Delta_0$
        are \emph{thick} and all other edges introduced by a
        delta composition are \emph{thin}.
	Observe the set of thick edges of $G\in \Delta_k$ is a perfect
        matching and therefore we deduce the following.

	\begin{LEM}\label{lem:leafedge}
          For graphs in $\Delta_k$,
          each leaf is incident only with a thick edge and no two
        leaves have a common neighbor. 
\end{LEM}

For  $G\in \Delta_k$,
	let $\mathcal{C}(G)$ be the set of triangles in $G$.
        First let us describe the set $\Theta(G)$ of marked vertices
        of $D_G$.
        For each thick edge $uv$ joining two non-leaf vertices, we
        have two new vertices $m(u,v)$ and $m(v,u)$ in $\Theta(G)$
        and for each pair of a vertex $v$ and a triangle $C$
        containing $v$, we have two new vertices $m(C,v)$ and
        $m(v,C)$.
        We will construct $D_G$ so that $V(D_G)$ is the disjoint union
        of $V(G)$ and $\Theta(G)$.
	For convenience, if $w$ is a leaf incident with an (thick) edge $vw$, 
	then $m(v,w):=w$.

	Now we describe all bags of $D_G$.
	For each vertex $v$ in $G$ of degree $n>1$,
        if $w$ is the unique neighbor of $v$ joined by a thick edge, then 
	let $B(v)$ be the graph isomorphic to $K_{1,(n-1)/2+1}$ on the vertex set 
        \[
	\{v,m(v,w)\}
	\cup \{m(v,C): C\in \mathcal{C}(G), v\in V(C)\}
	\]
	with the center $v$. 
	For each triangle $C$ of $G$,
	let $B(C)$ be the graph isomorphic to $K_3$ on the vertex set
        $\{m(C,v):v\in V(C)\}$.
	
	Let $D_G$ be the marked graph on the vertex set $\Theta(G)\cup V(G)$
	such that all bags of $D_G$ are  
	\[\{B(v):v\text{ is a non-leaf vertex in }G\} \cup \{B(C):C\in \mathcal{C}(G)\}\]
	and 
        the set $M(D_G)$ of all marked edges is exactly 
	\begin{multline*}
	\{m(v,C)m(C,v): C\in \mathcal{C}(G), v\in V(C)\}\\
	\cup\{m(v,w)m(w,v): vw \text{ is the thick edge joining
          two non-leaf vertices}\} 
	\end{multline*} 
	For a graph $G$ in $\Delta_2$, the marked graph $D_G$ is
        depicted in Figure~\ref{fig:canonical}.

	We now show that if $G\in \Delta_k$, then $D_G$ is the canonical decomposition of $G$.

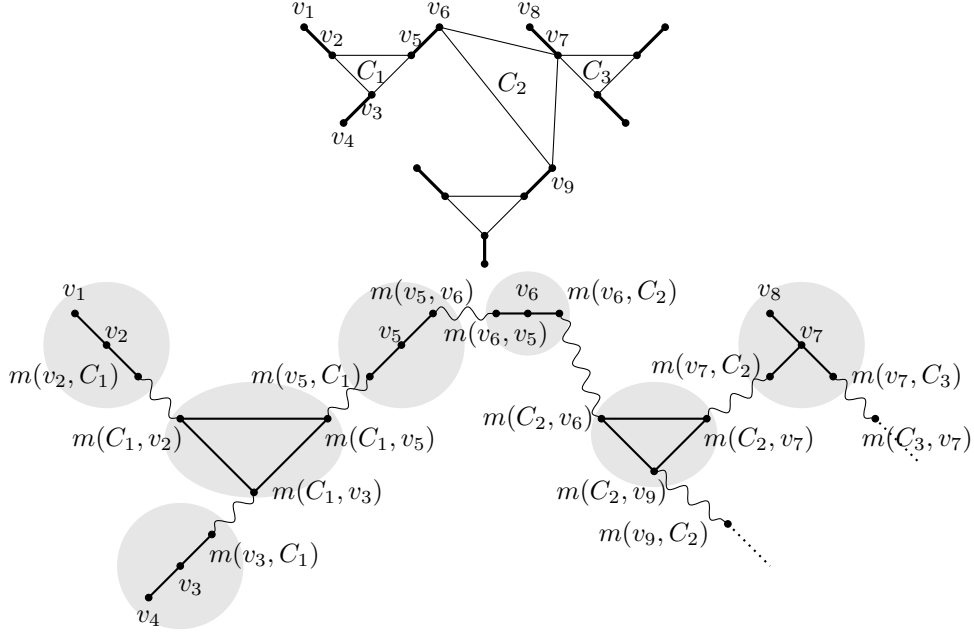
\begin{figure}\centering
 \tikzstyle{v}=[circle, draw, solid, fill=black, inner sep=0pt, minimum width=2.5pt]
 \tikzset{photon/.style={decorate, decoration={snake}}}
\begin{tikzpicture}[scale=0.075]

\draw (8,45) -- (22,45) -- (15, 38) -- (8, 45);
\node at (15,42) {$C_1$};

\draw (28,20) -- (42,20) -- (35, 13) -- (28, 20);

\draw (48,45) -- (62,45) -- (55, 38) -- (48, 45);
\node at (55,42) {$C_3$};

\draw[very thick] (8,45) -- (3, 50);
\node at (3,53) {$v_1$};
\node at (8,48) {$v_2$};

\draw[very thick] (22,45) -- (27, 50);
\node at (22,48) {$v_5$};
\node at (27,53) {$v_6$};

\draw[very thick] (15,38) -- (10, 33);
\node at (15,35) {$v_3$};
\node at (10,30) {$v_4$};

\draw[very thick] (28,20) -- (23, 25);
\draw[very thick] (42,20) -- (47, 25);
\draw[very thick] (35,13) -- (35, 8);
\draw[very thick] (48,45) -- (43, 50);
\node at (43,53) {$v_8$};
\node at (48,48) {$v_7$};

\node at (49,22) {$v_9$};

\draw[very thick] (62,45) -- (67, 50);
\draw[very thick] (55,38) -- (60, 33);

\draw (27,50) -- (47,25) -- (48,45) -- (27,50);
\node at (40,40) {$C_2$};

\node [v] at (8,45) {};
\node [v] at (3,50) {};
\node [v] at (22,45) {};
\node [v] at (27,50) {};
\node [v] at (15,38) {};
\node [v] at (10,33) {};
\node [v] at (28,20) {};
\node [v] at (23,25) {};
\node [v] at (42,20) {};
\node [v] at (47,25) {};
\node [v] at (35,13) {};
\node [v] at (35,8) {};

\node [v] at (48,45) {};
\node [v] at (43,50) {};
\node [v] at (62,45) {};
\node [v] at (67,50) {};
\node [v] at (55,38) {};
\node [v] at (60,33) {};

\end{tikzpicture}

\begin{tikzpicture}[scale=0.14]

\path [fill=gray!20] (-3,43) ellipse (8.5 and 5.5);
\draw[thick] (0-10,45) -- (14-10,45) -- (7-10, 38) -- (0-10, 45);
\node at (-15,43) {$m(C_1,v_2)$};
\node at (9,43) {$m(C_1,v_5)$};
\node at (4,38) {$m(C_1,v_3)$};

\path [fill=gray!20] (-17,52) circle (6);
\draw[thick] (-4-10, 49) -- (-7-10,52) -- (-10-10,55);
\node at (-21,49) {$m(v_2,C_1)$};
\node at (-16,53) {$v_2$};
\node at (-20,57) {$v_1$};

\draw[photon] (0-10,45) -- (-4-10, 49);

\path [fill=gray!20] (11,52) circle (6);
\draw[thick] (18-10, 49) -- (21-10,52) -- (24-10,55);

\node at (2,49) {$m(v_5,C_1)$};
\node at (10,53) {$v_5$};
\node at (13,57) {$m(v_5,v_6)$};

\path [fill=gray!20] (23,55) circle (4);
\draw[photon] (14-10,45) -- (18-10, 49);
\draw[photon] (24-10,55) -- (30-10, 55);
\draw[thick] (30-10,55) -- (33-10,55) -- (36-10,55);
\node at (20,53) {$m(v_6,v_5)$};
\node at (23,57) {$v_6$};
\node at (32,57) {$m(v_6,C_2)$};

\path [fill=gray!20] (-10,31) circle (6);
\draw[thick] (3-10, 34) -- (0-10,31) -- (-3-10,28);
\node at (-2,32) {$m(v_3,C_1)$};
\node at (-9,29) {$v_3$};
\node at (-13,26) {$v_4$};

\draw[photon] (7-10,38) -- (3-10, 34);

\node [v] at (0-10,45) {};
\node [v] at (-4-10,49) {};
\node [v] at (-7-10,52) {};
\node [v] at (-10-10,55) {};
\node [v] at (14-10,45) {};
\node [v] at (18-10,49) {};
\node [v] at (21-10,52) {};
\node [v] at (24-10,55) {};
\node [v] at (7-10,38) {};
\node [v] at (3-10,34) {};
\node [v] at (0-10,31) {};
\node [v] at (-3-10,28) {};

\node [v] at (30-10,55) {};
\node [v] at (33-10,55) {};
\node [v] at (36-10,55) {};

\path [fill=gray!20] (49,52) circle (6);
\path [fill=gray!20] (35,43.5) ellipse (6 and 5);

\draw[thick] (52, 49) -- (49,52) -- (46,55);
\draw[photon] (56,45) -- (52, 49);
\draw[thick,dotted] (60, 41) -- (56, 45);
\draw[thick] (49,52) -- (46,49);
\node at (59,49) {$m(v_7,C_3)$};
\node at (50,53) {$v_7$};
\node at (46,57) {$v_8$};
\node at (40,50) {$m(v_7,C_2)$};

\node at (60,43) {$m(C_3,v_7)$};

\node [v] at (56,45) {};
\node [v] at (52,49) {};
\node [v] at (49,52) {};
\node [v] at (46,55) {};

\node [v] at (46,49) {};

\node [v] at (42,35) {};

\draw[thick] (30,45) -- (40,45) -- (35, 40) -- (30, 45);
\node at (24,45) {$m(C_2,v_6)$};
\node at (45,43) {$m(C_2,v_7)$};
\node at (31,38) {$m(C_2,v_9)$};

\node [v] at (30,45) {};
\node [v] at (40,45) {};
\node [v] at (35,40) {};

\draw[photon] (30,45) -- (26, 55);
\draw[photon] (40,45) -- (46, 49);
\draw[photon] (35,40) -- (42, 35);
\draw[thick,dotted] (42,35) -- (46, 31);
\node at (35,34) {$m(v_9,C_2)$};

\end{tikzpicture}

\caption{A graph $G\in \Delta_2$ with thick edges, and a part of $D_G$.}\label{fig:canonical}
\end{figure}

\begin{PROP}
	For each graph $G\in \Delta_k$ with $k\ge 1$,
	the marked graph $D_G$ is the canonical decomposition of $G$.
\end{PROP}

\begin{proof}
	We first prove that $D_G$ is a split decomposition of $G$. 
	We use induction on $k$. %
	We may assume that $k\geq 2$ and let $C$ be the main triangle $v_1v_2v_3$ of $G$.  
	For each $1\le i\le 3$, 
	let $G_i$ be the component of $G\setminus v_1v_2\setminus v_2v_3\setminus v_3v_1$ 
	such that $v_i\in V(G_i)$,
	and let $D_{i}$ be the component of $D_G[V(D_G)\setminus \{m(C,v_1), m(C,v_2), m(C,v_3)\}]$ 
	such that $v_i\in V(D_{i})$.
	Let $w_i$ be the neighbor of $v_i$ such that $v_iw_i$ is thick.

        If $v_i$ is not a leaf in $G_i$, then 
        by construction, $D_i\setminus m(v_i,C)=D_{G_i}$.
        If $v_i$ is a leaf of $G_i$, then $D_{G_i}$ is obtained from
        $D_i\setminus m(v_i,C)$ by recomposing a marked edge joining
        $m(v_i,w_i)$ and $m(w_i,v_i)$.
        By the induction hypothesis, $D_{G_i}$ is a split
        decomposition of $G_i$ and therefore in both cases,
        $D_i\setminus m(v_i,C)$ is a split decomposition of $G_i$
        because we obtain $G_i$ from $D_i\setminus m(v_i,C)$ by
        recomposing all marked edges of $D_i\setminus m(v_i,C)$.

        Let $G_i'$ be the graph obtained from $D_i$ by recomposing all
        marked edges of $D_i$. 
        Then $m(v_i,C)$ is a leaf of $G_i'$
        and $G_i'\setminus m(v_i,C)=G_i$.

        If we recompose all marked edges of $D_G$ except three marked
        edges associated with $C$,
	then we obtain a marked graph obtained from the disjoint union
        of  
        $G_1'$, $G_2'$, $G_3'$, and $B(C)$
	by adding three marked edges in $\{m(v_i,C),m(C,v_i)\}_{1\le
          i\le 3}$.
        It is then clear that $G$ is obtained from this graph by recomposing three marked edges in
        $\{m(v_i,C),m(C,v_i)\}_{1\le i\le 3}$ from this graph.
        This proves that $D_G$ is a split decomposition of $G$.
 	
        It remains to check that $D_G$ is a canonical
        decomposition. 
	From the construction, every bag of $D_G$ is a complete bag or a star bag,	 
	and every star bag has marked vertices only on its leaves and no
        two complete bags are neighbors. This proves the lemma.
\end{proof}

\end{appendix}	
\end{document}